\documentclass[a4paper,12pt]{book}
\usepackage{t1enc}
\usepackage[latin2]{inputenc}
\usepackage{tikz,graphics,graphicx,color,booktabs}
\usepackage{caption,subcaption}
\usepackage{amsmath,amssymb,amsthm}
\usepackage{fancyhdr} 
\usepackage{multirow}
\usepackage{tikz}
\usetikzlibrary{arrows,positioning}

\usepackage{ifpdf}
  \usepackage{hyperref}
   \usepackage{epstopdf}

\usepackage{ifpdf}

\newtheorem{proposition}{Proposition}

\newtheorem{theorem}{Theorem}
\newtheorem{lemma}{Lemma}
\newtheorem{remark}{Remark}

\renewcommand   {\and}      {\qquad \textrm{and} \qquad }

\def            \eps        {\varepsilon}

\newcommand     \R          {\mathbb{R}}

\newcommand     \pa         {\partial}
\def            \to         {\rightarrow}

\newcommand\bfA{{\text{\textbf{A}}}}
\newcommand\bfF{{\text{\textbf{F}}}}

\newcommand\bfM{{\text{\textbf{M}}}}

\newcommand\bfu{{\text{\textbf{u}}}}

\newcommand\bfb{{\text{\textbf{b}}}}

\newcommand\bfw{{\text{\textbf{w}}}}
\renewcommand{\epsilon}{\varepsilon}
\renewcommand{\phi}{\varphi}
\renewcommand{\tilde}{\widetilde}

\begin{document}%

%
\pagenumbering{roman}%
\thispagestyle{empty}%
\begin{center}

	\large{\texttt{Master Thesis}}
	
	\vspace{2cm}
{\Huge \textbf{Abstract error analysis for Cahn--Hilliard type equations with dynamic boundary conditions}}

\vspace{3cm}
\large{\texttt{Eberhardt Karls University of T\"{u}bingen}}

\texttt{Department of Mathematics}

\vspace{2cm}

{submitted by}

\textsf{Paula Harder}

\vspace{2cm}

{supervised by}

\textsf{Prof. Dr. Christian Lubich}

\textsf{Dr. Bal\'{a}zs Kov\'{a}cs}

\vspace{2cm}

\today

\end{center}
\newpage

\chapter*{Acknowledgement}
Firstly, I would like to express my gratitude to Prof Dr. Lubich for giving me the possibility to write this master thesis. I would like to thank him for his supervision, for answering mathematical questions and giving me helpful advice. Beside offering me such an interesting topic I am very thankful for the office I could use, which provided a good surrounding for me to be productive.

 Moreover, I would like to thank Dr. Balázs Kovács for his great support during all the time of my thesis. I thank him for the vast amount of time he took for discussing questions with me and reading my texts and proofs. I am very grateful for the explanations, hints and ideas he gave me when I got stuck at some point.
 
Furthermore I would like to thank the whole numerical analysis group for the pleasant atmosphere they provided, for the lunch and tea breaks we had together.

I thank my brother for his very valuable help, for finding the bug in my implementation, for the possibility to always ask him for advice and his thorough proofreading. 

Additionally I would like to thank Niklas for the time we worked together every day in the C4 and his help at the end, Anja and Jonas for all the nice and restful tea breaks we had and Micheal for checking my English and for his constant backing.

Finally I would like to thank my parents for always supporting me incredibly.


\tableofcontents

\newpage

\thispagestyle{empty}

\chapter*{Introduction}

In this thesis we consider the Cahn--Hilliard equation with dynamic boundary conditions, in particular where a Cahn--Hilliard equation is prescribed on the boundary. That is for $u:\bar{\Omega}\times[0,T]
\to \R$
\begin{subequations}\label{c}
\begin{align}
\partial_t u &= \Delta (-\Delta u + W_\Omega'(u))\qquad &\text{in } \Omega, \label{eq:original_ch_1}\\
\partial_t u &= \Delta_\Gamma ( -\Delta_\Gamma u + W'_\Gamma(u)  + \partial_\nu u) - \partial_\nu ( -\Delta_\Gamma u + W'_\Gamma(u)  + \partial_\nu u) \qquad &\text{on}\  \Gamma. \label{eq:original_ch_rand1}
\end{align}
\end{subequations}

The Cahn--Hilliard equation, first described by \textit{Cahn\&Hilliard (1958)}, \cite{CH}, is an equation in mathematical physics and describes the phase separation of a binary fluid. The functions $W_\Omega,W_\Gamma :\mathbb{R} \to  \mathbb{R}$ are chemical potentials,  a typical example could be the double-well potential $W(u)=(u^2-1)^2$. {The solution  $u\in[-1,1]$ is the concentration of the two components of the fluid, $u=\pm 1$ then indicates a pure occurrence of one of the components.}
The two components are mixed in the beginning and then start building segregated domains, which are changing in time. To model the interactions with the wall in confined systems, dynamic boundary conditions are introduced, which is described in \textit{Goldstein et al. (2016)}, \cite{goldstein}.

This work addresses the problem of solving the Cahn--Hilliard equation numerically. For that we introduce an abstract formulation for Cahn--Hilliard type equations with dynamic boundary conditions, we conduct the spatial semidiscretization via finite elements and prove error bounds based on the technique of energy estimates. The variational formulation for Cahn--Hilliard/Cahn--Hilliard coupling, will apply to a larger abstract class of problems and is similar to the usual weak formulation of parabolic problems. In contrast to problems with non dynamic boundary conditions, the Hilbert spaces $L^2(\Omega)$ and $H^1(\Omega)$ are exchanged with the spaces $L^2(\Omega)\times L^2(\Gamma)$ and $\lbrace v\in H^1(\Omega): \gamma v \in H^1(\Gamma)\rbrace$, respectively.
Because we are considering a fourth-order differential equation, which will be described by a system of two second-order differential equations, the variational formulation also consists of a system of two equations.

The starting point of this thesis was a paper  by Balázs Kovács and Christian Lubich from 2016 on numerical analysis of parabolic problems with dynamic boundary conditions, \cite{1}. Our work has a similar structure and notations and uses some of the results shown in this paper. The Cahn--Hilliard equation with dynamic boundary conditions is mentioned in this paper, but no numerical analysis is conducted on it. This thesis will show  results for the Cahn--Hilliard/Cahn--Hilliard coupling, analogous to results shown in \textit{Kovács \& Lubich (2016)} for other parabolic problems with dynamic boundary conditions.

\textit{The thesis is organized as follows:}\\
In the Chapter \ref{c_1}  we consider a linear version of the Cahn--Hilliard equation with dynamic Cahn--Hilliard boundary conditions and derive a weak formulation for it. This weak formulation is then generalized, which yields an evolution equation in an abstract setting. Staying in this abstract setting, we use the technique of energy estimates to show a perturbation result, it bounds the error made by a perturbed solution of the weak formulation and will be later used to prove convergence.

Chapter 2 treats the semidiscretization in space by finite elements in an abstract setting. A consistency result is proved,  which together with the stability result from Chapter 1 then results in a convergence estimate. For proving consistency and convergence we collect  lemmata for bounding errors of the interpolation, the geometric approximation and the Ritz map.

The abstract setting as well as the convergence theorem are generalized in Chapter 3 to the nonlinear variant of the Cahn--Hilliard equation, by pointing out the main differences as compared to the linear theory. The proofs for stability and consistency will be partly similar to the linear case but more technical.

This work concludes with a description of the time discretization (Chapter 4), and presenting some results of numerical experiments, illustrating our theoretical results in Chapter 5. 
\newpage

\addcontentsline{toc}{chapter}{Introduction}

\newpage
\thispagestyle{empty}

\mainmatter
\pagenumbering{arabic}
\setcounter{page}{7}

\chapter{The Cahn--Hilliard equation with dynamic boundary conditions}\label{c_1}
\markboth{CHAPTER 1. THE CAHN--HILLIARD EQUATION WITH DYNAMIC B.C.}{chapter name}

We start this chapter by introducing a linear variant of the Cahn--Hilliard equation with Cahn--Hilliard boundary conditions. A weak formulation of it is derived, which fits into the abstract variational formulation we will present afterwards. In this abstract setting we will prove a perturbation result, which will later be applied to the semidiscrete case, and generalized  in the nonlinear case later.

By introducing a new function $w:\bar{\Omega}\to\R$ we rewrite the Cahn--Hilliard equation (\ref{c}) as a system of two second-order partial differential equations: For $u,w:\bar{\Omega}\to\R$
\begin{subequations}
	\begin{align}
	\partial_t u(x,t) &= \Delta w(x,t)\qquad&\text{in}\  \Omega\times[0,T]\\
	w(x,t) &=  -\Delta u(x,t) + W'_\Omega(u) \qquad&\text{in}\  \Omega\times[0,T],
	\end{align}
\end{subequations}
with dynamic Cahn--Hilliard boundary conditions
\begin{subequations}
	\begin{align}
	\partial_t u(x,t) &= \Delta_\Gamma w(x,t) - \partial_\nu w(x,t)\qquad &\text{on}\  \Gamma\times [0,T] \\
	w(x,t) &=  -\Delta_\Gamma u(x,t) + W'_\Gamma (u) + \partial_\nu u(x,t) \qquad &\text{on}\  \Gamma\times [0,T].
	\end{align}
\end{subequations}

 We assume that the domain $\Omega \subset \mathbb{R}^d$ is bounded and smooth, and define $\Gamma = \pa \Omega$ as its boundary. Furthermore $\pa_\nu$ denotes the normal derivative and $\nu$ the unit normal on $\Gamma$. The operator $\Delta_\Gamma$ is the Laplace--Beltrami operator, given by $\Delta_\Gamma u=\nabla_\Gamma\cdot \nabla_\Gamma u$ where $ \nabla_\Gamma (\gamma u) =(I-\nu \nu^T)\gamma (\nabla u)$ is the tangential gradient, $\nabla_\Gamma u $ instead of $\nabla_\Gamma (\gamma u) $ is written for brevity.

To consider a linear variant of the Cahn--Hilliard/Cahn--Hilliard coupling the derivatives of the chemical potentials,  $W',W_\Gamma' :\mathbb{R} \to  \mathbb{R}$, are replaced by the inhomogeneities $f_\Omega:\Omega\times [0,T]\to \mathbb{R}$ and $ f_\Gamma: \Gamma\times [0,T]\to \mathbb{R}$ respectively.
We obtain, for $u,w:\bar{\Omega}\to\R$,
\begin{subequations}\label{Cahn_H_bulk}
\begin{align}
\partial_t u(x,t) &= \Delta w(x,t)\label{eq:ch_1}\qquad&\text{in}\  \Omega\times[0,T]\\
w(x,t) &=  -\Delta u(x,t) + f_\Omega(x,t) \qquad&\text{in}\  \Omega\times[0,T],\label{eq:ch_2}
\end{align}
\end{subequations}
with dynamic linear Cahn--Hilliard boundary conditions
\begin{subequations}\label{Cahn_H_suface}
\begin{align}
\partial_t u(x,t) &= \Delta_\Gamma w(x,t) - \partial_\nu w(x,t)\qquad &\text{on}\  \Gamma\times [0,T] \label{eq:ch_rand1}\\
w(x,t) &=  -\Delta_\Gamma u(x,t) + f_\Gamma (x,t) + \partial_\nu u(x,t) \qquad &\text{on}\  \Gamma\times [0,T].\label{eq:ch_rand2}
\end{align}
\end{subequations}

\section{Weak formulation}

In this section we will derive a weak formulation of the linear Cahn--Hilliard equation with linear dynamic boundary conditions.

The first equation of the Cahn--Hilliard problem (\ref{eq:ch_1}) is multiplied with the test function $\phi^u\in \lbrace v\in H^1(\Omega) : \gamma v\in H^1(\Gamma)\rbrace$, the second equation with $\phi^w\in \lbrace v\in H^1(\Omega) : \gamma v\in H^1(\Gamma)\rbrace$ . Then both equations are integrated over the domain $\Omega$
\begin{eqnarray*}
	\int_\Omega\partial_t u \phi^u \mathrm{d}x &=& \int_\Omega \Delta w \phi^u \mathrm{d}x\\
	\int_\Omega w \phi^w \mathrm{d}x &=& - \int_\Omega \Delta u \phi^w \mathrm{d}x + \int_\Omega f_\Omega(x,t)\phi^w \mathrm{d}x.
\end{eqnarray*}
We denote by $\gamma v$ the trace of $v$ on $\Gamma$. On the terms on the right-hand side Green's formula is applied, therefore the above system changes into
\begin{eqnarray*}
	\int_\Omega \partial_t u \phi^u \mathrm{d}x &=& - \int_\Omega  \nabla w\cdot \nabla \phi^u \mathrm{d}x + \int_\Gamma \partial_\nu w \gamma \phi^u \mathrm{d}\sigma (x)\\
	\int_\Omega w \phi^w \mathrm{d}x &=& \int_\Omega \nabla u\cdot \nabla\phi^w \mathrm{d}x+ \int_\Omega f_\Omega\phi^w\mathrm{d}x- \int_\Gamma \partial_\nu u \gamma \phi^w\mathrm{d}\sigma(x).
\end{eqnarray*}
The dynamic boundary conditions (\ref{eq:ch_rand1}) and (\ref{eq:ch_rand2}) are now plugged into the equations
\begin{align*}
\int_\Omega \partial_t u \phi^u \mathrm{d}x &= - \int_\Omega  \nabla w\cdot \nabla \phi^u \mathrm{d}x + \int_\Gamma \Delta_\Gamma w \gamma \phi^u \mathrm{d}\sigma (x) - \int_\Gamma \gamma\partial_t u \gamma \phi^u \mathrm{d}\sigma (x)\\
\int_\Omega w \phi^w \mathrm{d}x &= \int_\Omega \nabla u\cdot \nabla\phi^w \mathrm{d}x+ \int_\Omega f_\Omega\phi^w\mathrm{d}x- \int_\Gamma \gamma w \gamma \phi^w\mathrm{d}\sigma(x)\\ &\qquad - \int_\Gamma \Delta_\Gamma u \gamma \phi^w\mathrm{d}\sigma(x)+ \int_\Gamma f_\Gamma \gamma \phi^w\mathrm{d}\sigma(x).
\end{align*}
Green's formula is applied again, now on the boundary, and we arrive at
\begin{align*}
\int_\Omega \partial_t u \phi^u \mathrm{d}x &= - \int_\Omega  \nabla w\cdot \nabla \phi^u \mathrm{d}x - \int_\Gamma \nabla_\Gamma w \cdot \nabla_\Gamma \phi^u \mathrm{d}\sigma (x) - \int_\Gamma \gamma\partial_t u \gamma \phi^u \mathrm{d}\sigma (x)\\
\int_\Omega w \phi^w \mathrm{d}x &= \int_\Omega \nabla u\cdot \nabla\phi^w \mathrm{d}x+ \int_\Omega f_\Omega\phi^w\mathrm{d}x- \int_\Gamma \gamma w \gamma \phi^w\mathrm{d}\sigma(x)\\ &\qquad + \int_\Gamma \nabla_\Gamma u \cdot \nabla_\Gamma \phi^w\mathrm{d}\sigma(x)+ \int_\Gamma f_\Gamma \gamma \phi^w\mathrm{d}\sigma(x).
\end{align*}
Rearranging the equations yields
\begin{subequations}\label{weak_ch}
\begin{align}
\Big(\int_\Omega & \partial_t u \phi^u \mathrm{d}x + \int_\Gamma\gamma \partial_t u \gamma \phi^u\mathrm{d}\sigma(x)\Big) +\Big( \int_\Omega  \nabla w\cdot \nabla \phi^u \mathrm{d}x + \int_\Gamma \nabla_\Gamma  w\cdot \nabla_\Gamma  \phi^u \mathrm{d}\sigma (x)\Big) &=0 
\end{align}
\begin{align}
\Big(\int_\Omega&  w \phi^w \mathrm{d}x+\int_\Gamma \gamma w \gamma \phi^w\mathrm{d}\sigma(x) \Big) -\Big(\int_\Omega \nabla u\cdot \nabla\phi^w \mathrm{d}x + \int_\Gamma \nabla_\Gamma u \cdot \nabla_\Gamma \phi^w\mathrm{d}\sigma(x) \Big)\nonumber \\ &= \int_\Omega f_\Omega\phi^w\mathrm{d}x  + \int_\Gamma f_\Gamma \gamma \phi^w\mathrm{d}\sigma(x).
\end{align}
\end{subequations}

This derivation yields the following lemma:
\begin{lemma}
	Every classical solution $(u, w)$ of the Cahn--Hilliard/Cahn--Hilliard coupling (\ref{Cahn_H_bulk})--(\ref{Cahn_H_suface}) is a solution of the weak formulation (\ref{weak_ch}).
\end{lemma}

For a sufficiently smooth solution of the weak formulation the reverse statement holds as well.
With the proof above and fundamental lemma of variational calculus can be shown:
\begin{remark}
	Every weak solution which is sufficiently regular is a strong solution.
\end{remark}

\section{Abstract setting}\label{abstract}

The abstract setting is the standard  abstract framework for parabolic problems, the difference is that we will have an equation system consisting of two equations. 
 
We consider the Gelfand triple $(V, H, V')$, as in \textit{Lubich\&Ostermann (1995)}, \cite[Section ~ 2.1]{ostermann},  with $(V,||\cdot||)$ and $(H,|\cdot|)$ being two Hilbert spaces. $V$ is continuously and densely embedded in $H$:
\vspace{0.5cm}
\begin{center}
\begin{tabular*}{5cm}[c]{ccccccc}
$V $& $\hookrightarrow $ & $H$  &$\cong$ & $H'$& $\hookrightarrow $ & $V'$\\
$K^{-1}||\cdot||$  & $\geq$ & $|\cdot|$ &$=$ & $|\cdot |$& $\geq$ & $K||\cdot ||_\ast$,
\end{tabular*}
\end{center}
with $K>0$.
\vspace*{0.5cm}

On $H$ we consider an inner product $m(\cdot,\cdot)$ which  induces the norm $|\cdot|$, whereas on $V$ we consider a continuous, symmetric bilinear form $a(\cdot,\cdot)$. We assume that this bilinear form satisfies the G\r{a}rding inequality: There exists an $\alpha > 0 $ and a $\mu\in \mathbb{R}$ such that
$$a(v,v)\geq \alpha ||v||^2-\mu|v|^2\qquad \forall v\in V. $$
The dual norm $||\cdot||_\ast$ on $V'$ is defined as follows
$$||\phi||_\ast := \sup_{v \in V\setminus\lbrace 0\rbrace}\frac{m(\phi,v)}{||v||}.$$
This definition implies following inequality, which  will be often used in the proof of the perturbation result, Proposition \ref{energy_est}
\begin{equation}\label{dual}
m(\phi, v) = ||v|| \Big(\frac{m(\phi,v)}{||v||}\Big)\leq ||\phi ||_\ast||v||.
\end{equation}
We introduce the energy semi-norm induced by $a(.,.)$
$$||v||_a^2:=a(v,v). $$ We consider another norm on $V$
$$||v||_{a^\ast}^2:= a^\ast(v,v),$$
where $$a^\ast(v,w):=a(v,w)+m(v,w),$$
and show its equivalence   to the $V$-norm $||\cdot||$.
\begin{lemma}\label{equi}
	There exists $c_0,c_1>0$ such that
	$$c_0||v||\leq ||v||_{a^\ast }\leq  c_1 ||v||, \qquad v \in V.$$
\end{lemma}
\begin{proof}
	The first inequality is shown by the G\r{a}rding inequality and the second one by using the continuity of the embedding and the bilinear form.
	Let $v \in V$,
	with the G\r{a}rding inequality we get $$\max(1,\mu)||v||_{a^\ast}^2\geq ||v||_a^2+\mu|v|^2\geq \alpha ||v||^2, $$
	which then yields $||v||_{a^\ast}\geq c_0 ||v|| $ for a suitable constant $c_0>0$.
	Because $V$ is continuously embedded in $H$ we have $|v|\leq K^{-1} ||v|| . $
	The continuity of $a(\cdot ,\cdot)$ gives use $||v||_a^2\leq M ||v||^2. $ for some $M>0$.
	Adding up both inequalities we obtain
	$||v||_{a^\ast}^2=||v||_a^2+|v|^2\leq c_1^2 ||v||^2$ 	for a suitable constant $c_1>0$.
\end{proof}
An abstract formulation for a problem like the Cahn--Hilliard equation with Cahn--Hilliard type dynamic boundary conditions is the following: Find $u\in C^1([0,T],H) \cap L^2(0,T;V)$ and $w \in L^2(0,T;V)$ such that
\begin{subequations}\label{eq:abstract_weak}
	\begin{align}
	m(\pa_t u(t), \phi^u)+a(w(t),\phi^u)&=0 \qquad &\forall \phi^u\in V \\
	m(w(t),\phi^w)-a(u(t) , \phi^w) &= m(f(t),\phi^w)\qquad&\forall \phi^w\in V\\
	u(0)  & = u_0 & \nonumber	
	\end{align}
\end{subequations}
for given initial data $u_0\in H$, an inhomogeneity $f \in L^2(0,T;H)$ and for time $0<t\leq T$.

\subsection*{The Cahn--Hilliard equation in the abstract setting} 
Now we will show how the Cahn--Hilliard equation with Cahn--Hilliard boundary conditions fits into the proposed abstract framework.\\
We define $V:=\lbrace v\in H^1(\Omega) : \gamma v\in H^1(\Gamma)\rbrace$ and $H:= L^2(\Omega)\times L^2(\Gamma)$. \\
The bilinear form and the inner product are set as 
\begin{align*}
m(u,v) &:= \int_\Omega uv \mathrm{d}x+\int_\Gamma \gamma u \gamma v\mathrm{d}\sigma(x)\\
a(u,v) &:= \int_\Omega  \nabla u\cdot \nabla v \mathrm{d}x + \int_\Gamma \nabla_\Gamma  u\cdot \nabla_\Gamma  v \mathrm{d}\sigma (x).
\end{align*}
Plugging these definitions into the abstract formulation \eqref{eq:abstract_weak}, we obtain the equation system \eqref{weak_ch}.

We define the following norms on $V$ and $H$
\begin{align*}
||v||^2 &:=a(v,v)+m(v,v) \qquad &\text{for } v\in V\\
|v|^2 &:=m(v,v) \qquad &\text{for } v\in H,
\end{align*}
We  then have $a(v,v)=||v||^2-|v|^2$, therefore $a(\cdot,\cdot)$ satisfies the G\r{a}rding inequality with $\alpha =1$ and $\mu=1$.\\
Furthermore we define a $H^2$-norm $$||u||_2^2:=||u||_{H^2(\Omega)}^2+||\gamma u||_{H^2(\Omega)}^2 .$$

\section{A continuous perturbation result}
In this section we will consider a perturbed solution, which solves the weak formulation with some defects $d^u(t),d^w(t) \in V$. We will show a perturbation result, which bounds the errors by the defects. This will be done in the abstract setting, as introduced in Section \ref{abstract}.

We consider the perturbed solution $(u^\ast, w^\ast )\in C^1([0,T],H) \cap L^2(0,T;V)\times L^2(0,T;V)$ of the system (\ref{eq:abstract_weak}). This perturbed solution is the exact solution of the following system  (omitting the argument $t$)
\begin{subequations}\label{per}
\begin{align}
m(\partial_t u^\ast, \phi^u)+a(w^\ast,\phi^u)&=m(d^u,\phi^u) \qquad &\forall \phi^u\in V\\
m(w^\ast,\phi^w)-a(u^\ast , \phi^w) &= m(f,\phi^w)+m(d^w,\phi^w)\qquad &\forall \phi^w\in V,
\end{align}
\end{subequations}
with some defects $d^u(t),d^w(t) \in V$.

Subtracting the equations of the system (\ref{eq:abstract_weak}), with exact solution $(u,w)$, of the equations of the perturbed system (\ref{per}) and setting $e^u:=u^\ast-u$ and $e^w:=w^\ast-w$ we obtain the error equation system
\begin{subequations}\label{abstract_defect}
\begin{align}
m(\partial_t e^u, \phi^u)+a(e^w,\phi^u)&=m(d^u,\phi^u) \qquad &\forall \phi^u\in V\\
m(e^w,\phi^w)-a(e^u , \phi^w) &= m(d^w,\phi^w)\qquad &\forall \phi^w\in V.
\end{align}
\end{subequations}
Using this error equation system and by energy estimates we will bound the errors $e^u$ and $e^w$ as follows.
\begin{proposition}\label{energy_est}
	Let   $(e^u, e^w) \in C^1([0,T];V)^2$ be the solution of the error equation system (\ref{abstract_defect}) with defects $d^u, d^w \in C^1([0,T];H)$, then the following estimate holds for all $t$ with $0<t\leq T$,
	\begin{align*}
	||&e^u(t)||^2+||e^w(t)||^2+\int_0^t||\partial_te^u(s)||^2\mathrm{d}s+\int_0^t||e^w(s)||^2\mathrm{d}s   \\ & 
	\leq C \Big( || e^u(0)||^2+||e^w(0)||^2+|| d^u(0)||_\ast^2+||d^u(t)||_\ast^2\\
	& \qquad +\int_0^t||d^u(s)||_\ast^2 \mathrm{d}s+\int_0^t||\partial_td^u(s)||_\ast^2\mathrm{d}s +\int_0^t||d^w(s)||_\ast^2\mathrm{d}s+\int_0^t||\partial_td^w(s)||_\ast^2\mathrm{d}s\Big),
	\end{align*}
	where $C$ depends on $\alpha, \mu$ and exponentially on $T$.
\end{proposition}

\begin{proof}
The proof is based on energy estimates, it is lengthy because we need bounds of the errors in the $V$-norm, which are point wise in time . The equations of the error equation system are tested repeatedly with different functions. The test functions are chosen in a way that, by using the anti-symmetric structure of (\ref{abstract_defect}), with addition or subtraction either the term with $m(\cdot,\cdot)$ on the right-hand side or the term with $a(\cdot,\cdot)$ cancels out. 
The proof is structured in three parts, the first two parts both derive an estimate, which are then combined in the third part.

	In this proof the argument $t$ is omitted.\\
	\textit{(i) First Energy Estimate:}
	
	We consider (\ref{abstract_defect}). By setting $\phi^u=e^u$ and $\phi^w=e^w$ as test functions we obtain
	\begin{subequations}
		\begin{align}
		m(\partial_t e^u, e^u)+a(e^w,e^u)&=m(d^u,e^u)\label{eq:error1} \\
		m(e^w,e^w)-a(e^u ,e^w) &=m(d^w,e^w).\label{error2}
		\end{align}
	\end{subequations}
	Addition of (\ref{eq:error1}) and (\ref{error2}) and symmetry of $a(\cdot,\cdot)$ then yields
	\begin{equation}\label{error3}
	m(\partial_t e^u,e^u)+m(e^w,e^w) = m(d^u,e^u)+m(d^w,e^w).
	\end{equation}
	On the other hand, choosing
	 $\phi^u=e^w$ and $\phi^w=\partial_t e^u$  yields
	\begin{subequations}
		\begin{align}
		m(\partial_t e^u, e^w)+a(e^w,e^w)&=m(d^u,e^w)\label{error4} \\
		m(e^w,\partial_t e^u)-a(e^u ,\partial_t e^u) &= m(d^w,\partial_t e^u).\label{error5}
		\end{align}
	\end{subequations}
	By subtraction of (\ref{error4}) and (\ref{error5}) and symmetry of $m(\cdot,\cdot)$ we obtain
	\begin{equation}\label{error6}
	a(e^u ,\partial_t e^u)+a(e^w,e^w) =m(d^u,e^w)- m(d^w,\partial_t e^u).
	\end{equation}
	Now the equations (\ref{error3}) and (\ref{error6}) are summed up, we then use that for the bilinear form $a^\ast(\cdot,\cdot)$ the following holds
	\begin{align}\label{bili_derivatiom}
	a^\ast(\pa_t v,v) =\frac{1}{2}\frac{\mathrm{d}}{\mathrm{d}t}||v||_{a^\ast}^2,
	\end{align}
	and we obtain
	\begin{equation*}
	\frac{1}{2}\frac{\mathrm{d}}{\mathrm{d}t}||e^u||_{a^\ast}^2+||e^w||_{a^\ast}^2 =m(d^u,e^u)+m(d^u,e^w)- m(d^w,\partial_t e^u)+ m(d^w,e^w).
	\end{equation*}
	With (\ref{dual}) the right-hand side is estimated as follows:
	\begin{equation*}
	\frac{1}{2}\frac{\mathrm{d}}{\mathrm{d}t}||e^u||_{a^\ast}^2+||e^w||_{a^\ast}^2 \leq ||d^u||_\ast||e^u||+||d^u||_\ast||e^w||+||d^w||_\ast||\partial_t e^u||+||d^w||_\ast||e^w||.
	\end{equation*}
	Application of the $\epsilon$-Young inequality yields
	\begin{equation*}
	\frac{1}{2}\frac{\mathrm{d}}{\mathrm{d}t}||e^u||_{a^\ast}^2+||e^w||_{a^\ast}^2 \leq \frac{\epsilon '}{2}||e^u||^2+\frac{\epsilon}{2}||\partial_t e^u||^2+\epsilon ' ||e^w||^2+\frac{1}{\epsilon '}||d^u||_\ast^2+\frac{1}{2\epsilon '}||d^w||_\ast^2+\frac{1}{2\epsilon }||d^w||_\ast^2,
	\end{equation*}
	for $\eps,\eps'>0$.\\
	We integrate from $0$ to $t$ and apply the equivalence of the norms (Lemma \ref{equi}) to obtain
	\begin{align*}
	\frac{c_0}{2}||e^u(t)||^2+(c_0-\epsilon')\int_0^t ||e^w(s)||^2\mathrm{d}s   & \leq \frac{c_1}{2}||e^u(0)||^2+\frac{\epsilon'}{2}\int_0^t||e^u(s)||^2\mathrm{d}s \\
	&\qquad  +\frac{\epsilon}{2}\int_0^t||\partial_t e^u(s)||^2\mathrm{d}s
	+\frac{1}{2\epsilon'}\int_0^t||d^u(s)||_\ast^2\mathrm{d}s\\
	& \qquad +\Big(\frac{1}{2\epsilon'}+\frac{1}{2\epsilon}\Big)\int_0^t||d^w(s)||_\ast^2\mathrm{d}s.
	\end{align*}
	We now introduce the constant $c>0$. This constant, whenever it is used in this or following proofs, is a generic constant in may change in every step.
	
	To absorb the integral term with $e_h^w$ we choose $\eps':=\frac{c_0}{2}$, which then yields
		\begin{align*}
		||e^u(t)||^2+\int_0^t ||e^w(s)||^2\mathrm{d}s  & \leq c\Big(||e^u(0)||^2+\int_0^t||e^u(s)||^2\mathrm{d}s+{\epsilon}\int_0^t||\partial_t e^u(s)||^2\mathrm{d}s
		\\
		&\qquad +\int_0^t||d^u(s)||_\ast^2\mathrm{d}s +\Big(1+\frac{1}{\epsilon}\Big)\int_0^t||d^w(s)||_\ast^2\mathrm{d}s  \Big).
		\end{align*}
	By using the Gronwall inequality we obtain
	  \begin{align*}
	  ||&e^u(t)||^2+\int_0^t ||e^w(s)||^2\mathrm{d}s \\  & \leq c\Big(||e^u(0)||^2+\epsilon\int_0^t||\partial_t e^u(s)||^2\mathrm{d}s+\nonumber
	  \int_0^t||d^u(s)||_\ast^2\mathrm{d}s +\Big(1+\frac{1}{\epsilon}\Big)\int_0^t||d^w(s)||_\ast^2\mathrm{d}s  \Big).
	  \end{align*}
	  Setting $\eps:=\frac{1}{2c}$ yields
	   \begin{align}\label{ee1}
	   ||e^u(t)||^2+\int_0^t ||e^w(s)||^2\mathrm{d}s   & \leq c\Big(||e^u(0)||^2+\nonumber
	   \int_0^t||d^u(s)||_\ast^2\mathrm{d}s\\
	   & \qquad  +\int_0^t||d^w(s)||_\ast^2\mathrm{d}s \Big) +\frac{1}{2}\int_0^t||\partial_t e^u(s)||^2\mathrm{d}s.
	   \end{align}
	\textit{(ii) Second Energy Estimate:}\\
	To control the last term on the right-hand side of (\ref{ee1}) we will show an estimate with the same term  on the left-hand side. Therefore
	the second equation of the error equation system is differentiated with respect to time,  and the following system is obtained
	\begin{subequations}\label{error_der}	
	\begin{align}
	m(\partial_t e^u, \phi^u)+a(e^w,\phi^u)&=m(d^u,\phi^u) \\
	m(\partial_t e^w,\phi^w)-a(\partial_t e^u , \phi^w) &= m(\partial_td^w,\phi^w).
	\end{align}	
	This part of the proof is structured in three subparts, the first part derives an estimate from the equation (\ref{error_der}a), the second part derives one from the equation (\ref{error_der}b) and the third part combines these estimates. This structure allows to have intermediate results, which then is used for the nonlinear case as well.\\
	\end{subequations}
	\textit{ Part a:}\\
	We test (\ref{error_der}a) with $\partial_t e^u$, (\ref{error_der}b) with $e^w$ and add them up to obtain
		\begin{eqnarray*}
			m(\partial_t e^u, \partial_t e^u)+m(\partial_t e^w,e^w)&=&m(d^u,\partial_t e^u)+m(\partial_td^w,e^w).
		\end{eqnarray*}
		For the left-hand side of the equation we use $m(\pa_t v,v) =\frac{1}{2}\frac{\mathrm{d}}{\mathrm{d}t}|v|^2 $ and for the first term on the right-hand side (\ref*{dual}), which then yields
		\begin{align*}
		\frac{\mathrm{d}}{\mathrm{d}t}|e^w|^2+|\partial_te^u|^2 & \leq ||d^u||_\ast||\partial_te^u||+m(\partial_td^w,e^w).
		\end{align*}
		We integrate from $0$ to $t$ and obtain
		\begin{align*}
	\frac{1}{2}|e^w(t)|^2+	\int_0^t|\pa_t e^u(s)|^2\mathrm{d}s &\leq \frac{1}{2}|e^w(0)|^2 + \int_0^t||d^u(s)||_{\ast}||\pa_t e^u(s)||\mathrm{d}s\\
		& \qquad +\int_0^tm(\pa_td^w(s),e^w(s))\mathrm{d}s.
		\end{align*}
		The $\eps$-Young inequality yields
		\begin{align}\label{ee2an}
		|e^w(t)|^2+\int_0^t|\pa_t e^u(s)|^2\mathrm{d}s&\leq |e^w(0)|^2+\eps\int_0^t||\pa_t e^u(s)||^2\mathrm{d}s +\nonumber \frac{1}{\eps}\int_0^t||d^u(s)||_{\ast}^2\mathrm{d}s\\
		& \qquad +2\int_0^tm(\pa_td^w(s),e^w(s))\mathrm{d}s.
		\end{align}
		This estimate is an intermediate stage from which we will continue in the nonlinear case.
		
		 For the linear case we now estimate the last term with (\ref{dual}) and the Young inequality
		\begin{align}\label{ee2a}
	|e^w(t)|^2 +	\int_0^t|\pa_t e^u(s)|^2\mathrm{d}s&\leq |e^w(0)|^2 +\eps\int_0^t||\pa_t e^u(s)||^2\mathrm{d}s+\int_0^t||e^w(s)||^2\mathrm{d}s\nonumber\\
	& \qquad+ \frac{1}{\eps}\int_0^t||d^u(s)||_{\ast}^2\mathrm{d}s +\int_0^t||\pa_td^w(s)||_\ast^2\mathrm{d}s .
		\end{align}
	\textit{Part b:}\\
	We test (\ref{error_der}a) with $\partial_t e^w$, (\ref{error_der}b) with $\pa_t e^u$ and subtract them to obtain	
	\begin{eqnarray*}
		a(e^w,\partial_t e^w)+a(\partial_t e^u , \partial_t e^u)&=&m(d^u,\partial_t e^w)-m(\partial_td^w,\partial_t e^u).
	\end{eqnarray*}	
	Because $\partial_t e^w$ cannot be absorbed we rewrite $m(d^u,\partial_t e^w)$, by using the product rule, as follows:
	\begin{equation*}
	m(d^u,\partial_t e^w)=\frac{\mathrm{d}}{\mathrm{d}t}m(d^u,e^w)-m(\partial_td^u,e^w),
	\end{equation*}
	therefore the equation above changes into
	\begin{eqnarray*}
		a(e^w,\partial_t e^w)+a(\partial_t e^u , \partial_t e^u)&=&\frac{\mathrm{d}}{\mathrm{d}t}m(d^u,e^w)-m(\partial_td^u,e^w)-m(\partial_td^w,\partial_t e^u).
	\end{eqnarray*}
	With $a(\pa_t v,v) =\frac{1}{2}\frac{\mathrm{d}}{\mathrm{d}t}||v||_a^2 $ and integration we obtain
	\begin{align*}
	\frac{1}{2}||e^w(t)||_{a}^2+\int_0^t||\pa_t e^u(s)||_{a}^2\mathrm{d}s &= \frac{1}{2}||e^w(0)||_{a}^2 +m(d^u(t),e^w(t))-m(d^u(0),e^w(0))\\
	& \qquad- \int_0^t m(\partial_td^u(s),e^w(s))\mathrm{d}s-\int_0^tm(\partial_td^w(s),\partial_t e^u(s))\mathrm{d}s.
	\end{align*}
	This is estimated as follows
	\begin{align*}
	\frac{1}{2}||e^w(t)||_{a}^2+\int_0^t||\pa_t e^u(s)||_{a}^2\mathrm{d}s &\leq \frac{1}{2}||e^w(0)||_{a}^2 +m(d^u(t),e^w(t))+|m(d^u(0),e^w(0))|\\
	& \qquad+\Big| \int_0^t m(\partial_td^u(s),e^w(s))\mathrm{d}s\Big|\\
	& \qquad+\Big|\int_0^tm(\partial_td^w(s),\partial_t e^u(s))\mathrm{d}s\Big|.
	\end{align*}
	We estimate the inner product $m(\cdot,\cdot)$ with (\ref{dual}), which yields
	\begin{align*}
 \frac{1}{2}||e^w(t)||_{a}^2+\int_0^t||\pa_t e^u(s)||_{a}^2\mathrm{d}s &\leq \frac{1}{2}||e^w(0)||_{a}^2 + ||d^u(0)||_{\ast}||e^w(0)||+||d^u(t)||_{\ast}||e^w(t)||\\
	& \qquad + \int_0^t||d^u(s)||_{\ast}||e^w(s)||\mathrm{d}s\\
	& \qquad +\Big|\int_0^tm(\pa_td^w(s),\pa_t e^u(s))\mathrm{d}s\Big|.
	\end{align*}
	With Young's and $\eps$-Young's inequality we obtain
	\begin{align}\label{ee2bn}
	||e^w(t)||_{a}^2+\int_0^t||\pa_t e^u(s)||_{a}^2\mathrm{d}s &\leq ||e^w(0)
	||_{a}^2 +||e^w(0)||^2+\eps'||e^w(t)||^2+ ||d^u(0)||_{\ast}^2\nonumber\\
	& \qquad +\frac{1}{\eps '}||d^u(t)||_{\ast}^2+\int_0^t||e^w(s)||_h^2\mathrm{d}s+ \int_0^t||d^u(s)||_{\ast}^2\mathrm{d}s\nonumber\\
	& \qquad +\Big|\int_0^tm(\pa_td^w(s),\pa_t e^u(s))\mathrm{d}s\Big|.
	\end{align}
	This intermediate estimate will be used for the nonlinear case. 
	
	We now continue to estimate the last term. By (\ref{dual}) and the $\eps$-Young inequality we obtain
	\begin{align}\label{ee2b}
	||e^w(t)||_{a}^2+\int_0^t||\pa_t e^u(s)||_{a}^2\mathrm{d}s &\leq ||e^w(0)||_{a}^2+||e^w(0)||^2 +\eps'||e^w(t)||^2+\nonumber ||d^u(0)||_{\ast}^2\\
	& \qquad +\frac{1}{\eps' }||d^u(t)||_{\ast}^2+\eps \int_0^t||\pa_t e^u(s)||^2\mathrm{d}s+\nonumber \int_0^t||e^w(s)||^2\mathrm{d}s\\
	& \qquad\int_0^t||d^u(s)||_{\ast}^2\mathrm{d}s +\frac{1}{\eps}\int_0^t||\pa_td^w(s)||_\ast^2\mathrm{d}s.
	\end{align}
	
	\textit{Part c: Combination.}
	
	We add the inequalities (\ref{ee2a}) and (\ref{ee2b}) and obtain
	\begin{align*}
	\int_0^t||\pa_t e^u(s)||_{a^\ast}^2\mathrm{d}s+ ||e^w(t)||_{a^\ast}^2 &\leq ||e^w(0)||_{a^\ast}^2+||e^w(0)||^2 +\eps'||e^w(t)||^2+\nonumber ||d^u(0)||_{\ast}^2\\
	& \qquad+\frac{1}{\eps' }||d^u(t)||_{\ast}^2 +2\eps \int_0^t||\pa_t e^u(s)||^2\mathrm{d}s\\
	& \qquad +2\int_0^t||e^w(s)||^2\mathrm{d}s+\nonumber \Big(1+\frac{1}{\eps}\Big)\int_0^t||d^u(s)||_{\ast}^2\mathrm{d}s\\
	& \qquad  +\Big(1+\frac{1}{\eps}\Big)\int_0^t||\pa_td^w(s)||_\ast^2\mathrm{d}s.
	\end{align*}
	The equivalence of the norms $||\cdot||_{a^\ast}$  and $||\cdot||$ (Lemma \ref{equi}) is applied, which yields
	\begin{align*}
	c_0\int_0^t||\pa_t e^u(s)||^2\mathrm{d}s+ c_0||e^w(t)||^2 &\leq c_1||e^w(0)||^2+||e^w(0)||^2 +\eps'||e^w(t)||^2+\nonumber ||d^u(0)||_{\ast}^2\\
	& \qquad+\frac{1}{\eps' }||d^u(t)||_{\ast}^2 +2\eps \int_0^t||\pa_t e^u(s)||^2\mathrm{d}s\\
	& \qquad +2\int_0^t||e^w(s)||^2\mathrm{d}s+\nonumber \Big(1+\frac{1}{\eps}\Big)\int_0^t||d^u(s)||_{\ast}^2\mathrm{d}s\\
	& \qquad  +\Big(1+\frac{1}{\eps}\Big)\int_0^t||\pa_td^w(s)||_\ast^2\mathrm{d}s.
	\end{align*}
	To absorb the third and sixth term we choose $\eps':=\frac{c_0}{2}$ and $\eps:=\frac{c_0}{4}$ and we obtain
	\begin{align*}
	||e^w(t)||^2 +\int_0^t||\pa_t e^u(s)||^2\mathrm{d}s &\leq c\Big(||e^w(0)||^2 +\nonumber ||d^u(0)||_{\ast}^2+||d^u(t)||_{\ast}^2\\
	& \qquad +\int_0^t||e^w(s)||^2\mathrm{d}s +\nonumber \int_0^t||d^u(s)||_{\ast}^2\mathrm{d}s +\int_0^t||\pa_td^w(s)||_\ast^2\mathrm{d}s\Big).
	\end{align*}
	The Gronwall inequality yields
	\begin{align}\label{ee2}
	||e^w(t)||^2 +\int_0^t||\pa_t e^u(s)||^2\mathrm{d}s &\leq c\Big(||e^w(0)||^2 +\nonumber ||d^u(0)||_{\ast}^2+||d^u(t)||_{\ast}^2\\
	& \qquad + \int_0^t||d^u(s)||_{\ast}^2\mathrm{d}s +\int_0^t||\pa_td^w(s)||_\ast^2\mathrm{d}s\Big).
	\end{align}
	\textit{(iii) Combining the energy estimates}
	
	Adding the inequalities (\ref{ee1}) and (\ref{ee2}) yields 
	\begin{align}\label{ee}
	||&e^u(t)||^2+ ||e^w(t)||^2+\nonumber\int_0^t||\pa_t e^u(s)||^2\mathrm{d}s+ \int_0^t ||e^w(s)||^2\mathrm{d}s  \\ & \leq c\Big(||e^u(0)||^2+||e^w(0)||^2+||d^u(0)||_{\ast}^2+||d^u(t)||_{\ast}^2+\nonumber
	\int_0^t||d^u(s)||_\ast^2\mathrm{d}s\\
	& \qquad  +\int_0^t||d^w(s)||_\ast^2\mathrm{d}s+\int_0^t||\pa_td^w(s)||_\ast^2\mathrm{d}s \Big) +\frac{1}{2}\int_0^t||\partial_t e^u(s)||^2\mathrm{d}s.
	\end{align}
	With the absorption of the last term on the right-hand side we obtain the stated result.
	\end{proof}
	\begin{remark}
With a similar but easier proof by energy estimates the solution $(u, w)\in C^1([0,T];V)^2$  of the weak formulation (\ref{eq:abstract_weak}) with $f \in C^1([0,T];H)$ satisfies the bound
	\begin{align*}
	||&u(t)||^2+||w(t)||^2+\int_0^t||\partial_t u(s)||^2 \mathrm{d}s  +\int_0^t||w(s)||^2\mathrm{d}s\\ & \leq ||u(0)||^2+||w(0)||^2+\int_0^t ||f(s)||_\ast^2\mathrm{d}s+4\int_0^t||\partial_t f(s)||_\ast^2 \mathrm{d}s.
	\end{align*}
\end{remark}

\newpage
\chapter{Semidiscrete error analysis}
In this section we study the spatial discretization with linear finite elements of the linear Cahn--Hilliard equation with dynamic boundary conditions. We present a semidiscrete weak formulation, which still fits into the abstract framework introduced in the former chapter. Therefore the stability of the semidiscretization follows from the perturbation result Proposition \ref{energy_est}. The bulk and the surface are both triangulated, the spaces $V$ and $H$ are replaced by finite dimensional spaces, the inner product and the bilinear form by semidiscrete analogues. To prove convergence in this setting we combine stability bounds with consistency estimates.  We will present several error bounds for interpolation, the Ritz map and the geometric approximation before. The stability analysis conducted in this chapter can be applied to the linear variant of the Cahn--Hilliard equation with Cahn--Hilliard boundary conditions, but we will work within the abstract formulation, which makes it possible to use this results for other equations as well.

The domain $\Omega$ is assumed to be smooth and is approximated by a polyhedral domain $\Omega_h$, which is triangulated with a mesh width  $h$. The triangulation $\mathcal{T}_h$ is supposed to be quasi-uniform. The surface $\Gamma$ is approximated by the boundary of $\Omega_h$, $\Gamma_h:=\pa \Omega_h$. The domain $\Omega_h$ is constructed such that the vertices of $\Gamma_h$ are on $\Gamma$.  For each node $x_k$ there is a basis function $\phi_k$ such that $\phi_j(x_k)=\delta_{jk}$ for $j,k=1,\ldots,N$. The basis functions are continuous and linear on each  element. We define $V_h:=\langle \phi_1,\phi_2,\ldots,\phi_N\rangle$ as the span of the nodal basis functions.  The restrictions of the basis functions to the boundary   
span the boundary finite element space, i.e. $S_h=\langle \gamma_h \phi_1,\ldots,\gamma_h \phi_N\rangle$, with the trace operator $\gamma_h$ on $\Gamma_h$.
The discretization of the second Hilbert space $H$ is the same as for $V$, $H_h:=V_h$.

This semidiscrete variational formulation is a special case of the abstract variational formulation.  
The inner product and bilinear form on $V_h$ are denoted by $m_h(\cdot,\cdot)$ and $a_h(\cdot,\cdot)$ and have the same properties as $m(\cdot,\cdot)$ and $a(\cdot,\cdot)$ in the abstract formulation.   The continuous norms $|\cdot|$ and $||\cdot||$ are replaced by $||\cdot||_h$ and $|\cdot|_h$, where we define for $v_h\in V_h$
$$|v_h|^2_h:=m_h(v_h,v_h) $$
and
$$||v_h||^2_h:=a_h(v_h,v_h). $$

We define the following notations of discrete semi-norms and norms on $V_h$ analogous to the continuous case:
$$||v||_{a,h}^2:=a_h(v_h,v_h)\qquad \text{for } v_h\in V_h $$
and the semidiscrete dual norm
\begin{align}\label{dual_norm}
||\phi||_{\ast,h} := \sup_{v \in V\setminus \lbrace 0\rbrace}\frac{m_h(\phi,v)}{||v||_h}.
\end{align}

Finally to formulate the semidiscrete variational equation system, we need a finite element approximation of the function $f(\cdot,t)\in  H$. This will be done via the finite element interpolation which, for a continuous function $f$, is defined by $$\tilde{I}_h f(\cdot,t):= \sum_{i=1}^{N}f(x_i,t)\phi_i \quad \in V_h. $$
The weak formulation can then be written as follows: Find $(u_h,w_h)\in C^1([0,T],V_h)\times L^2(0,T;V_h)$ such that the following holds
\begin{subequations}\label{weak_semi}
\begin{align}
m_h(\partial_t u_h, \phi^u_h)+a_h(w_h,\phi^u_h)&=0 \qquad & \forall \phi^u_h\in V_h,\\
m_h(w_h,\phi^w_h)-a_h(u_h , \phi^w_h) &= m_h(\tilde{I}_hf,\phi^w_h) \qquad & \forall \phi^w_h\in V_h,
\end{align}
\end{subequations}
with initial condition
\begin{align*}
m_h(u_h(0),\phi_h^u) & = m_h(\tilde{I}_hu_0,\phi_h^u) \qquad & \forall \phi^u_h\in V_h,
\end{align*}
for initial values $u_0\in H$.

We present $u_h(t)$ and $w_h(t)$ with respect to the basis $\phi_1,\phi_2,\ldots,\phi_N$ as
$$ u_h(t)=\sum_{i=1}^N u_i(t)\phi_i,\qquad w_h(t)=\sum_{i=1}^N w_i(t)\phi_i.$$
With the time dependent coefficient vectors $\bfu(t)=(u_1(t),u_2(t),\ldots,u_N(t))^T\in\R^N $ and $\bfw(t)=(w_1(t),w_2(t),\ldots,w_N(t))^T\in\R^N $ we have the following matrix-vector formulation
\begin{subequations}\label{m_v_formulation}
\begin{eqnarray}
\bfM\dot{\text{\textbf{u}}}(t) + \bfA\bfw(t) &=& 0\\
\bfM \bfw(t) - \bfA\bfu(t) &=& \bfb(t),
\end{eqnarray}
\end{subequations}
with $\bfM_{ij}:=m_h(\phi_i,\phi_j), \bfA_{ij}:= a_h(\phi_i,\phi_j), \bfb_i(t):= m_h(\tilde{I}_h f(t),\phi_i)$.
This is a system of ordinary differential equations coupled with a system of linear equations, a differential algebraic equation system. 

For the semidiscrete weak formulation of the linear Cahn--Hilliard equation with Cahn--Hilliard boundary conditions $m_h(\cdot,\cdot)$ and $a_h(\cdot,\cdot)$ are defined as follows:
\begin{align*}
m_h(u_h,v_h) &:= \int_{\Omega_h} u_hv_h \mathrm{d}x+\int_{\Gamma_h} (\gamma_h u_h) (\gamma_h v_h)\mathrm{d}\sigma_h\\
a_h(u_h,v_h) &:= \int_{\Omega_h}  \nabla u_h\cdot \nabla v_h \mathrm{d}x + \int_{\Gamma_h} \nabla_{\Gamma_h}  u_h\cdot \nabla_{\Gamma_h}  v_h \mathrm{d}\sigma_h,
\end{align*} 
where $\nabla_{\Gamma_h}$, the discrete tangential gradient, is defined in a piecewise sense by $\nabla_{\Gamma_h} (\gamma v_h)   =(I-\nu_h\nu_h^T)\gamma (\nabla v_h)$, where $\nu_h$ is the unit outward normal on $\Gamma_h$. For brevity we write $\nabla_{\Gamma_h} v_h$ instead of $\nabla_{\Gamma_h} (\gamma v_h)$. $\sigma_h$ is the discrete surface measure.

To be able to compare a semidiscrete solution to the exact solution we transform the semidiscrete solution into a function defined on $\Omega$, therefore we \textit{lift} it to the surface.
We define a lift of a function  $v_h:\Gamma_h \to \mathbb{R}$ on the surface by $v_h^l:\Gamma \to \mathbb{R}$ with $v_h^l(p)=v_h(x)$, where $x-p$ is orthogonal to the tangent space of $\Gamma$. Such a point x is unique, see \textit{Dzuik\&Elliot (2007)}, \cite{dande}. For a function on the bulk $v_h:\Omega_h \to \mathbb{R}$   we define its lift $v_h^l:\Omega \to \mathbb{R}$, via $v_h^l:=v_h\circ G_h^{-1}$.  The function $G_h:\Omega_h \to \Omega$ is defined piecewise for every element $E\in\mathcal{T}_h$ 
$$ G_h|_E (x) = F_e((F_E)^{-1}(x)) \qquad x\in E.$$
The function $F_E$ is the standard affine linear map between the reference element and $E$. The function $F_e$  maps from the reference element onto the smooth element $e\subset \Omega$ and is $\mathcal{C}^1$. For details, see  \textit{Elliot\&Ranner (2013)}, \cite[Section]{er13}. On the surface both definitions coincide.

Now it is possible to formulate error bounds for the errors between the exact solution of the abstract problem \eqref{eq:abstract_weak} and the solution of its semidiscretization \eqref{weak_semi}. We state two estimates in different norms for the weak formulation of the Cahn--Hilliard equation with dynamic boundary conditions. For the error in the $V$-norm we obtain a first-order bound and in the $H$-norm an optimal second-order bound. For both bounds we  require that $||u_h(0)-\tilde{R}_h u(0)||\leq Ch^2$ and $||w_h(0)-\tilde{R}_h w(0)||\leq Ch^2$, where $(\tilde{R}_h)^l$ is the Ritz map introduced later, see Section \ref{sec_ritz}. Therefore we will choose our initial values $u_h(0)=\tilde{R}_h u(0)$ and $ w_h(0)=\tilde{R}_h w(0) $.
\begin{theorem}\label{convergence}
	Let $(u,w)\in C^2([0,T],\lbrace v\in H^2(\Omega) : \gamma v\in H^2(\Gamma)\rbrace)\times C^1([0,T],\lbrace v\in H^2(\Omega) : \gamma v\in H^2(\Gamma)\rbrace)$ be the solution of (\ref{Cahn_H_bulk} )--(\ref{Cahn_H_suface}), with $f\in C^{1}([0,T],\lbrace v\in H^2(\Omega) : \gamma v\in H^2(\Gamma)\rbrace)$. Then, for sufficiently small $h\leq h_0$ the semidiscrete solution of (\ref{weak_semi})
	$(u_h,w_h)\in C^1([0,T],V_h)^2$ satisfies the following error bounds for all $t$ with $0\leq t\leq T$,
	\begin{align*}
	||&w_h^l(t)-w(t)||^2+||u_h^l(t)-u(t)||^2  \\ \qquad & +\int_0^t||\partial_tu_h^l(s)-\partial_t u(s)||^2\mathrm{d}s +\int_0^t||w_h^l(s)-w(s)||^2\mathrm{d}s  \leq C h^2.
	\end{align*}
	In the weaker norm, $|\cdot |$, we obtain a second-order error bound
	\begin{align*}
	|&w_h^l(t)-w(t)|^2+|u_h^l(t)-u(t)|^2  \\ \qquad & +\int_0^t|\partial_tu_h^l(s)-\partial_t u(s)|^2\mathrm{d}s +\int_0^t|w_h^l(s)-w(s)|^2\mathrm{d}s  \leq C h^4.
	\end{align*}
With
\begin{align*}
C 
& =C'\Big(||u(t)||_2^2+ ||\partial_t u (0)|| ^2+ ||\partial_t u (t)|| ^2 +||w(t)||_2^2+ \int_0^t ||\partial_t u(s)||_2^2 \mathrm{d}s + \int_0^t||\partial_{tt} u(s)||_2^2 \mathrm{d}s\\ & \qquad+ \int_0^t||w(s)||_2^2\mathrm{d}s +\int_0^t||\partial_t w(s)||_2^2\mathrm{d}s+\int_0^t||f(s)||_2 ^2\mathrm{d}s+\int_0^t||\partial_t f(s)||_2 ^2 \mathrm{d}s 
\Big),
\end{align*}
where $C'>0$  depends exponentially on $T$ and is independent from $h$ and $t$.
\end{theorem}
This convergence theorem is concluded after the subsequent sections: A stability estimate, in Section \ref{stab}, some preparational lemmata concerning geometric approximation, in Section \ref{prep}, and a consistency estimate, in Section \ref{cons}.

\section{Stability}\label{stab}
We consider a projected exact solution $(u_h^\ast,w_h^\ast)\in C^1([0,T],V_h)\times L^2(0,T;V_h)$, which is a projection of the solution of the abstract weak formulation \eqref{eq:abstract_weak} onto the finite element space. This projected exact solution will be later chosen as images of the Ritz map. These perturbed solution is the exact solutions of the following system  (omitting the argument $t$)
\begin{subequations}\label{per_semi}
	\begin{align}
	m_h(\partial_t u_h^\ast, \phi_h^u)+a_h(w_h^\ast,\phi_h^u)&=m_h(d_h^u,\phi_h^u) \qquad &\forall \phi_h^u\in V_h\label{per_semi_a}\\
	m_h(w_h^\ast,\phi_h^w)-a_h(u_h^\ast , \phi_h^w) &= m_h(\tilde{I}_hf,\phi_h^w)+m_h(d_h^w,\phi_h^w)\qquad &\forall \phi_h^w\in V_h,
	\end{align}
\end{subequations}
with some defects $d_h^u(t),d_h^w(t) \in V_h$.

Subtracting the equations of the exact system (\ref{weak_semi}) with the equations of the perturbed system (\ref{per_semi}) and setting $e_h^u:=u_h^\ast-u_h$ and $e_h^w:=w_h^\ast-w_h$ we obtain the error equation system:
\begin{subequations}\label{error_semi}
	\begin{align}
	m_h(\partial_t e_h^u, \phi^u_h)+a_h(e_h^w,\phi^u_h)&=m_h(d_h^u,\phi_h^u) \qquad & \forall \phi^u_h\in V_h\\
	m_h(e_h^w,\phi^w_h)-a_h(e_h^u , \phi^w_h) &= m_h(d_h^w,\phi_h^w) \qquad & \forall \phi^w_h\in V_h.
	\end{align}
\end{subequations}
 The proof of the continuous perturbation result, Proposition \ref{energy_est}, shows the following bound: 
\begin{proposition}\label{energy_est_semidiscr}
	Let $(e^u_h, e^w_h) \in C^1([0,T];V_h)^2$ be the solution of the semidiscrete error equation system \ref{error_semi} with defects $d^u_h, d^w_h\in C^1([0,T];V_h)$, then the following estimate holds for all $t$ with $0<t\leq T$,
	\begin{align*}
	||&e_h^u(t)||_h^2+||e_h^w(t)||_h^2+\int_0^t||\partial_te_h^u(s)||_h^2\mathrm{d}s+\int_0^t||e_h^w(s)||_h^2\mathrm{d}s  \\  & 
	\leq C\Big( || e_h^u(0)||_h^2+||e_h^w(0)||_h^2 +|| d_h^u(0)||_{\ast,h}^2+||d_h^u(s)||_{\ast,h}^2\\
	& \qquad+\int_0^t||d_h^u(s)||_{\ast,h}^2 \mathrm{d}s+\int_0^t||\partial_td_h^u(s)||_{\ast,h}^2\mathrm{d}s+\int_0^t||d_h^w(s)||_{\ast,h}^2\mathrm{d}s+\int_0^t||\partial_td_h^w(s)||_{\ast,h}^2\mathrm{d}s\Big),
	\end{align*}
	with $C>0$ depending on $\alpha $ and $\mu$ and exponentially on $T$, but independent from $h$ and $t$.
\end{proposition}
\section{Preliminary results}\label{prep}
To be able to prove Theorem \ref{convergence} we first collect error bounds related to surface and boundary approximation. Starting with the error made by interpolation, followed by estimates on geometric approximation errors, we will conclude this preparational  section by introducing the Ritz map and present its error bounds. These are proven for the setting of the Cahn--Hilliard equation with dynamic Cahn--Hilliard boundary condition. 
\subsection{Interpolation}
We consider the piecewise linear finite element interpolation operator $\tilde{I}_h$, which was introduced in the beginning of this chapter. Now we consider its lift and prove error bounds for that lifted interpolation.
For the interpolation $\tilde{I}_hv \in V_h$ of a continuous function $v$ we define $I_hv:=(\tilde{I}_hv )^l \in V_h^l$. 
\begin{lemma}[Interpolation errors]\label{inter}
	Let $v\in  H^2(\Omega)$ such that $\gamma v \in H^2(\Gamma)$. Then we have the following error estimates for interpolation in the bulk and respectively for interpolation on the surface
	\begin{align}
	&||v-I_hv||_{L^2(\Omega)}+h||\nabla(v-I_hv)||_{L^2(\Omega)}  \leq {C}h^2 ||v||_{H^2(\Omega)}\label{interpolation_bulk} \\
		 &||\gamma(v-I_hv)||_{L^2(\Gamma)}+h||\nabla_\Gamma(v-I_hv)||_{L^2(\Gamma)}  \leq {C}h^2 ||\gamma v||_{H^2(\Gamma)}.\label{interpolation_surface} 
		 \end{align}
		 For $v\in L^\infty (\Omega)$ such that $(\gamma v)\in L^\infty (\Gamma)$ we obtain
		 \begin{align}
		&||v-I_hv||_{L^\infty(\Omega)}  \leq {C}||v||_{L^\infty(\Omega)}\label{inter_infty_bulk} \\
		& ||\gamma(v-I_hv)||_{L^\infty(\Gamma)}  \leq {C}||\gamma v||_{L^\infty(\gamma)}\label{interpolation_infty_surface}
		.
	\end{align}
\end{lemma}

For the proof of (\ref{interpolation_bulk}) we refer to \textit{Elliot\&Ranner (2013)}, \cite[Propositiom ~ 5.4]{er13}, for (\ref{interpolation_surface}) to \textit{Dzuik (1988)}, \cite{d88}, for (\ref{inter_infty_bulk}) to \textit{Brenner \& Scott}\cite[Section ~ 4.4]{brenner},  and for (\ref{interpolation_infty_surface}) to \textit{Demlow (2013)}, \cite[Proposition ~2.7]{demlow}.

By adding estimates \eqref{interpolation_bulk} and \eqref{interpolation_surface} from Lemma \ref{inter} we obtain
\begin{equation}\label{interpolation_error}
|v-I_hv|\leq Ch^2||v||_2,
\end{equation}
and 
\begin{equation}\label{interpolation_error2}
||v-I_hv||\leq Ch||v||_2.
\end{equation}
\subsection{Geometric approximation error}
Due to the variational crime, $V_h\nsubseteq V$, another type of error has to be taken into account, the error made by the discrete versions of the inner product and the bilinear form.\\
The layer of lifted elements which have a boundary face is denoted by $B_h^l\subseteq \Omega$, see  \textit{Elliot\&Ranner (2013)}, \cite[Section 4.2 ]{er13}.
\begin{lemma} For all $v\in H^1(\Omega)$ the following estimate holds:
	$$||v||_{L^2(B_h^l)}\leq Ch^{\frac{1}{2}}||v||_{H^1(\Omega)} $$
\end{lemma}
For the proof see \textit{Elliot \&Ranner(2013)} \cite[ Lemma 6.3]{er13}. 

\begin{lemma}[Geometric Approximation Error]\label{lemma_geom} 
	Let $v_h, w_h\in V_h$ then the following holds
	\begin{align}
		|a(v^l_h,w_h^l)-a_h(v_h,w_h)| & \leq Ch||\nabla v_h^l||_{L^2(B_h^l)}||\nabla w_h^l ||_{L^2(B_h^l)}\nonumber\\ & \quad +Ch^2(||\nabla v_h^l ||_{L^2(\Omega)}||\nabla w_h^l ||_{L^2(\Omega)} +||\nabla_\Gamma v_h^l ||_{L^2(\Gamma)}||\nabla_\Gamma w_h^l ||_{L^2(\Gamma)})\label{geom}\\
		|m(v^l_h,w_h^l)-m_h(v_h,w_h)| & \leq Ch||v_h^l||_{L^2(B_h^l)}||w_h^l ||_{L^2(B_h^l)}\nonumber\\ & \quad +Ch^2(||v_h^l ||_{L^2(\Omega)}||w_h^l ||_{L^2(\Omega)} +||\gamma v_h^l ||_{L^2(\Gamma)}||\gamma w_h^l ||_{L^2(\Gamma)})\label{geom2}
	\end{align}
\end{lemma}
For the proof see \textit{Kovács \& Lubich (2016)}, \cite[Lemma 3.9]{1}.

Combining both lemmata gives the following second-order estimate, which is the one we will use in the consistency proof
\begin{align}\label{geom_appr_error}
\nonumber|&m(v^l_h,w_h^l)-m_h(v_h,w_h)|\\ \nonumber & \leq ch||v_h^l||_{L^2(B_h^l)}||w_h^l ||_{L^2(B_h^l)} +ch^2(||v_h^l ||_{L^2(\Omega)}||w_h^l ||_{L^2(\Omega)} +||\gamma v_h^l ||_{L^2(\Gamma)}||\gamma w_h^l ||_{L^2(\Gamma)})\\
\nonumber& \leq ch^2||v_h^l||_{H^1(\Omega)}||w_h^l ||_{H^1(\Omega)} +ch^2(||v_h^l ||_{L^2(\Omega)}||w_h^l ||_{L^2(\Omega)} +||\gamma v_h^l ||_{L^2(\Gamma)}||\gamma w_h^l ||_{L^2(\Gamma)})\\
\nonumber& \leq ch^2(||v_h^l||_{H^1(\Omega)}||w_h^l ||_{H^1(\Omega)} +(||v_h^l ||_{L^2(\Omega)} +||\gamma v_h^l ||_{L^2(\Gamma)})(||w_h^l ||_{L^2(\Omega)}+||\gamma w_h^l ||_{L^2(\Gamma)}))\\
\nonumber& \leq ch^2(||v_h^l||_{H^1(\Omega)}||w_h^l ||_{H^1(\Omega)} +| v_h^l |
|w_h^l |)
\\
& \leq ch^2 ||v_h^l||\ ||w_h^l ||.
\end{align}
With the lemma on geometric approximation we will show that the norms $||\cdot||_h$ and $||\cdot||$, $|\cdot|_h$ and $|\cdot|$ are $h$-uniformly equivalent.
\begin{lemma}[Equivalence of norms]\label{equivalence} 
	Let $u_h\in V_h$, then there exist $c_i>0$ for $i=0,\ldots 7$, such that the following estimate holds
	\begin{align}
		 &(1-hc_0) ||u_h^l|| \leq ||u_h^l||_h \leq(1+h c_1) ||u_h^l||\\
		 &(1-hc_2) |u_h^l|_h  \leq |u_h| \leq (1+hc_3) |u_h^l|
	\end{align}
	For $h$ sufficiently small this implies the equivalence of the norms
	\begin{align}
	 &	c_4 ||u_h||_h  \leq ||u_h^l|| \leq c_5 ||u_h||_h\\
	 &	c_6 |u_h|_h  \leq |u_h^l| \leq c_7 |u_h|_h.
	\end{align}
	
\end{lemma}
This result was already stated in \textit{Kovács\&Lubich (2016)}, \cite[Equation 3.13]{1}, we give the complete proof here.
\begin{proof}
	The proof is using the lemma on the geometric approximation error and the triangle inequality.
	With Lemma \ref{lemma_geom} and by setting $w_h:=v_h$ we obtain
	\begin{align*}
	\big| ||v^l_h||_a^2-||v_h||_{h,a}^2\big|  \leq ch||\nabla v_h^l||_{L^2(B_h^l)}^2+ch^2(||\nabla v_h^l ||_{L^2(\Omega)}^2 +||\nabla_\Gamma v_h^l ||_{L^2(\Gamma)}^2)\leq ch||v_h^l||_a^2,
	\end{align*}
	and
	\begin{align*}
	\big||v^l_h|^2-|v_h|_{h}^2\big| \leq ch||v_h^l||_{L^2(B_h^l)}^2+ch^2(|| v_h^l ||_{L^2(\Omega)}^2 +||\gamma v_h^l ||_{L^2(\Gamma)}^2) \leq ch|v_h^l|^2.
	\end{align*}
	By switching the signs of the two terms on the left-hand side this implies
	$$ \big||v_h||_{h,a}^2-||v^l_h||_a^2\big|\leq ch|v_h^l|_a^2$$
	and
	$$\big||v^l_h|^2-|v_h|_h^2\big|\leq ch|v_h^l|^2 $$
	Using the triangle inequality we obtain the estimate in the energy norm and the $H$-norm
	$||v_h||_{a,h}^2\geq (1-ch)||v_h^l||^2_{a}$ and $||v_h||_{a,h}^2\leq (1+ch)||v_h^l||^2_{a},$
$|v_h|_h^2\geq (1-ch)|v_h^l|^2$ and $|v_h|_h^2\leq (1+c)|v_h^l|^2.$\\
Adding up the inequalities which are upon each other gives the estimate in the $V$-norm
	$$||v_h||_h^2\geq (1-ch)||v_h^l||^2$$ and $$||v_h||_h^2\leq (1+ch)||v_h^l||^2.$$
\end{proof}
\subsection{The Ritz map}\label{sec_ritz}

The Ritz map is an almost a-orthogonal map onto the finite element space. We define it here as in \textit{Kovács\&Lubich (2016)}, \cite[Section 3.4]{1}.

The Ritz map $R_h:V\rightarrow V_h^l$ is defined in two steps:
We determine a solution of the following equation
$$a_h(\tilde{R}_hu, \phi_h) = a(u, \phi_h^l). $$
The \emph{Ritz map} $\tilde{R}_hu:V\to V_h$ is then lifted to obtain the final Ritz map $$R_hu:=(\tilde{R}_h u)^l\in V_h^l.$$
The error made by the Ritz map is bound as follows
\begin{lemma} Let $u\in  H^2(\Omega)$ such that $\gamma u \in H^2(\Gamma)$. The error of the Ritz map in the energy norm satisfies the following first-order bound
	\begin{equation}\label{ritz1}
	||R_h u - u||_a^2 \leq C h^2 ||u||_2^2.
	\end{equation}
The error of the Ritz map in the $|\cdot |$-norm is given by
		\begin{equation}\label{ritz2}
		|R_h u - u|^2 \leq C h^4 ||u||_2^2.
		\end{equation}
\end{lemma}
For the proof we refer to \textit{Kovács\& Lubich (2016)}, Lemma 3.11 and Lemma 3.13.

Adding the inequalities (\ref{ritz1}) and (\ref{ritz2}) yields an error bound for the Ritz error in the $V$-norm
\begin{equation}\label{ritz_v}
||R_hu-u||^2\leq C h^2 ||u||_2^2.
\end{equation}
\section{Consistency}\label{cons}
In this section we will use the preparational lemmata to show a second-order error bound for the defects and their derivatives.

\begin{proposition}\label{pr_consist}
	Let $(u,w)$ be a solution of (\ref{Cahn_H_bulk})--(\ref{Cahn_H_suface}) with inhomogeneity $f$ such that $(u,w)$ and $f$ satisfy the regularity assumptions of Theorem \ref{convergence}. We define the perturbed solution $u_h^\ast:=R_hu$, $w_h^\ast:= R_hw$ using the Ritz map.
	Then there exist defects $d_h^u, d_h^w\in C^1([0,t];V_h)$ such that 
	$(u_h^\ast,w_h^\ast)$ is a solution of the system (\ref{per_semi}).
	These defects satisfy
	\begin{align}
		||d_h^u(t)||_{\ast,h}^2 &\leq c_0 h^4  \qquad \text{and} \qquad||\pa_t d_h^u(t)||_{\ast,h}^2 \leq c_1 h^4, \\
		||d_h^w(t)||_{\ast,h}^2 &\leq c_2 h^4 \qquad \text{and} \qquad ||\pa_t d_h^w(t)||_{\ast,h}^2 \leq c_3 h^4 ,		
	\end{align}
	with $c_0= c||\pa_t u||_2^2, c_1= c||\pa_{tt} u||_2^2, c_2=c(||w||_2^2+||f||_2^2), c_3=c(||\pa_t w||_2^2+||\pa_t f||_2^2) $ and $c>0$.
\end{proposition}
\begin{proof}
	We start with the existence of the defects  $d_h^u, d_h^w\in C^1([0,t];V_h)$.
	We briefly fix the time $t$.
	The map 
$$T^u:V_h\to\R,\qquad\phi_h^u\mapsto m_h(\partial_t u_h^\ast,\phi_h^u)+a_h(w_h^\ast,\phi_h^u)$$
is a linear and continuous functional on $V_h$.
	Since $V_h=H_h$, this is also a linear and continuous functional on $H_h$. 
	Since $m_h$ is an inner product on the Hilbert space $H_h$,
	we can apply the Riesz representation theorem.
	Thus there exists an element $d_h^u(t)\in H_h=V_h$
	such that $T^u(\phi_h^u)=m_h(d_h^u(t),\phi_h^u)$ holds for all $\phi_h^u\in V_h=H_h$.
	Thus there exists $d_h^{u}\in V_h$ that satisfies (\ref{per_semi_a}). By defining $$T^w:V_h\to\R,\qquad\phi_h^u\mapsto m_h( w_h^\ast,\phi_h^u)-a_h(u_h^\ast,\phi_h^u)-m_h( \tilde{I}_h f,\phi_h^u),$$ we obtain existence of $d_h^w\in V_h$ with the same arguments.
		Since $w_h^\ast$ and $\tilde I_hf$ are continuously differentiable in time, $u_h^\ast$ is twice continuously differentiable in time,
		it follows from (\ref{per_semi}) that $d_h^u$ and $d_h^w$ are differentiable in time.
		
The rest of this proof consists of two main parts, (i) for estimating the defect in $u$, $d_h^u$, (Part a) and its derivative (Part b) and (ii) for estimating the defect in $w$, $d_h^w$, (Part a) and its derivative (Part b). We decompose the defects in a way that we can estimate them by using the lemmata from the previous section.

In this proof we use the generic constant $c>0$ and mention that it is independent from $h$ for this proof and all following appearances.

\textit{(i) Part a: Defect in $u$.}\\
For the defect in $u$, $d_h^u$, we use the equation (\ref{error_semi}a) and obtain, for $\phi_h^u\in V_h$
\begin{eqnarray*}
	m_h(d_h^u,\phi_h^u) &=& m_h(\partial_t e_h^u, \phi^u_h)+a_h(e_h^w,\phi^u_h)\\
	&=&m_h(\tilde{R}_h\partial_t u-\partial_t u_h,\phi_h^u)+a_h(\tilde{R}_hw-w_h,\phi_h^u)\\
	&=& m_h(\tilde{R}_h\partial_t u,\phi_h^u)-m_h(\partial_t u_h,\phi_h^u)+a_h(\tilde{R}_hw,\phi_h^u)-a_h(w_h,\phi_h^u).
\end{eqnarray*}	
		Using the definition of the Ritz map then yields
		\begin{eqnarray*}
		m_h(d_h^u,\phi_h^u) 
		&=& m_h(\tilde{R}_h\partial_t u,\phi_h^u)-m_h(\partial_t u_h,\phi_h^u)+a(w,(\phi_h^u)^l)-a_h(w_h,\phi_h^u).
	\end{eqnarray*}
 Using the equation (\ref{weak_semi}a) we obtain
		\begin{eqnarray*}
			m_h(d_h^u,\phi_h^u)
			&=& m_h(\tilde{R}_h\partial_t u_,\phi_h^u)+a(w,(\phi_h^u)^l).
		\end{eqnarray*}
			Again we insert an equation from the weak formulation for the last term on the right-hand side, but now the continuous variant  (\ref{eq:abstract_weak}a), which yields
			\begin{eqnarray}\label{def_1}
				m_h(d_h^u,\phi_h^u) 
				&=& m_h(\tilde{R}_h\partial_t u,\phi_h^u)-m(\partial_t u,(\phi_h^u)^l).
			\end{eqnarray}
	We add and subtract $m((\tilde{R}_h \partial_t u)^l-R_h \partial_tu,(\phi_h^u)^l))$, then  for the defect in $u$ we obtain
		\begin{equation}\label{defect_1}
		m_h(d_h^u,\phi_h^u)= (m_h(\tilde{R}_h\partial_t u,\phi_h^u)-m((\tilde{R}_h \partial_t u)^l,(\phi_h^u)^l))+m(R_h \partial_tu-\partial_t u,(\phi_h^u)^l).
		\end{equation}
		We will start estimating the right-hand side.
	Using the Cauchy--Schwarz inequality we obtain
	\begin{align*}
	m_h(d_h^u,\phi_h^u)\leq \Big| m_h(\tilde{R}_h\partial_t u,\phi_h^u)-m((\tilde{R}_h \partial_t u)^l,(\phi_h^u)^l)\Big|+|R_h \partial_tu-\partial_t u|\ |(\phi_h^u)^l|.
	\end{align*}
	The geometric approximation error (\ref{geom_appr_error}) and the error of the Ritz map  (\ref{ritz2}) then yield
	\begin{align*}
	m_h(d_h^u,\phi_h^u)\leq ch^2||(\tilde{R}_h \partial_t u)^l || \ ||(\phi_h^u)^l||+ch^2||\partial_tu||_2\ |(\phi_h^u)^l|.
	\end{align*}
	We use that $|\phi|\leq ||\phi||$, and together with the $h$-uniformly equivalence of the norms $||\cdot||_h$ and $||\cdot ||$, for sufficiently small $h$, it is obtained that
	\begin{align*}
	m_h(d_h^u,\phi_h^u)\leq ch^2||{R}_h \partial_t u || \ ||\phi_h^u||_h+ch^2||\partial_tu||_2\ ||\phi_h^u||_h.
	\end{align*}
	By the definition of the dual norm \eqref{dual_norm} we estimate the defect in the dual norm in the following:
	\begin{align*}
	||d_h^u||_{\ast,h} &\leq  \sup_{\phi_h^u \in V_h\setminus \lbrace 0\rbrace}\frac{ch^2||{R}_h \partial_t u || \ ||\phi_h^u||_h+ch^2||\partial_tu||_2\ ||\phi_h^u||_h}{||\phi_h^u||_h}\\
	& = ch^2 (||R_h\pa_t u|| + ||\pa_t u||_2).
	\end{align*}
To  eliminate $R_h$ on the right hand side we add and subtract $\partial_t u$ and use the error of the Ritz map (\ref{ritz_v}), which then yields, for sufficiently small $h$
	\begin{align*}
	||d_h^u||_{\ast,h}
	& \leq c h^2 (|| R_h\partial_t u-\partial_t u|| +||\partial_t u||+||\partial_t u||_2)\\
	& \leq c h^2 (h ||\partial_t u||_2 +||\partial_t u||+||\partial_t u||_2)\\
	& \leq ch^2 ||\partial_t u||_2.
	\end{align*}
\textit{Part b: Derivative of the defect in u.}\\
	For estimating the derivative of the defect we take the equation for the defect from (\ref{def_1}), differentiate the equation with respect to time, use that $\pa_t(R_hu)=R_h(\pa_t u)$ and add and subtract the term $m((\tilde{R}_h \partial_{tt} u)^l,(\phi_h^u)^l))$, which yields
	\begin{equation}\label{der_def_1}
	m_h(\partial_t d_h^u,\phi_h^u) = (m_h(\tilde{R}_h\partial_{tt} u,\phi_h^u)-m((\tilde{R}_h \partial_{tt} u)^l,(\phi_h^u)^l))-m(\partial_{tt} u-R_h \partial_{tt}u,(\phi_h^u)^l).
	\end{equation}
	A similar proof as above then yields
	\begin{align*}
	||\pa_t d_h^u||_{\ast,h}
	& \leq ch^2 ||\partial_{tt} u||_2.
	\end{align*}
	\textit{(ii)Part a: Defect in $w$.}
	
We consider the defect in $w$. Let $d_h^w$ be the defect in $w$ of the equation system (\ref{error_semi}). We use the equation (\ref{error_semi}b) and obtain
\begin{align*}
	m_h(d_h^w,\phi_h^w)&= m_h(e^w_h,\phi_h^w)-a_h(e^u_h,\phi_h^w)\\
	&= m_h(\tilde{R}_hw-w_h,\phi_h^w)-a_h(\tilde{R}_hu-u_h,\phi_h^w)\\
	&= m_h(\tilde{R}_hw,\phi_h^w)-m_h(w_h,\phi_h^w)-a_h(\tilde{R}_hu,\phi_h^w)+a_h(u_h,\phi_h^w).
	\end{align*}
	The definition of the Ritz map yields
	\begin{align*}
	m_h(d_h^w,\phi_h^w)
	&= m_h(\tilde{R}_hw,\phi_h^w)-m_h(w_h,\phi_h^w)-a(u,(\phi_h^w)^l)+a_h(u_h,\phi_h^w).
	\end{align*}
	The last term is exchanged using the equation (\ref{weak_semi}b)
	of the semidiscrete weak formulation and afterwards the second to last term by using the continuous weak formulation. We obtain
		\begin{align}\label{def_2}
		m_h(d_h^w,\phi_h^w)
		&= m_h(\tilde{R}_hw,\phi_h^w)-m_h(w_h,\phi_h^w)-a(u,(\phi_h^w)^l)+a_h(u_h,\phi_h^w)\nonumber\\
		&=m_h(\tilde{R}_hw,\phi_h^w)-m_h(w_h,\phi_h^w)-a(u,(\phi_h^w)^l)-m_h(\tilde{I}_hf,\phi_h^w)+m_h(w_h,\phi_h^w) \nonumber\\	
		&= m_h(\tilde{R}_hw,\phi_h^w)-a(u,(\phi_h^w)^l)-m_h(\tilde{I}_hf,\phi_h^w)\nonumber\\
		&= m_h(\tilde{R}_hw,\phi_h^w)-m(w,(\phi_h^w)^l)+m(f,(\phi_h^w)^l)-m_h(\tilde{I}_hf,\phi_h^w).
		\end{align}
		We add and subtract $m((\tilde{R}_hw)^l,(\phi_h^w)^l) $ and $m(I_hf,(\phi_h^w)^l) $ to finally obtain the following equation for the defect in $w$
		\begin{align}\label{defect_2}
		\nonumber m_h(d_h^w,\phi_h^w) &= (m_h(\tilde{R}_hw,\phi_h^w)-m((\tilde{R}_hw)^l,(\phi_h^w)^l))+m(R_hw-w,(\phi_h^w)^l)\\ & \quad+(m(f-I_hf,(\phi_h^w)^l))+(m(I_hf,(\phi_h^w)^l)-m_h(\tilde{I}_hf,\phi_h^w)).
		\end{align}
	We will estimate the left-hand side of the above equation.
	We apply the Cauchy-Schwarz inequality to the second and third term, which yields
	\begin{align*}
	\nonumber m_h(d_h^w,\phi_h^w) &\leq (m_h(\tilde{R}_hw,\phi_h^w)-m((\tilde{R}_hw)^l,\phi_h^w))+|R_hw-w|\ |(\phi_h^w)^l|\\ & \quad+|f-I_hf|\ |(\phi_h^w)^l)|+(m(I_hf,(\phi_h^w)^l)-m_h(\tilde{I}_hf,\phi_h^w)).
	\end{align*}
	For the first and last term we use the geometric approximation error (\ref{geom_appr_error}), for the second the error of the Ritz map \ref{ritz2}, and for the third term the interpolation error (\ref{interpolation_error}), which altogether yield
	\begin{align}
		\nonumber m_h(d_h^w,\phi_h^w) &\leq ch^2(||(\tilde{R}_hw)^l||\ ||(\phi_h^w)^l||+||w||_2\ |(\phi_h^w)^l|+||f||_2|(\phi_h^w)^l)|+||{I}_hf||\ ||(\phi_h^w)^l||).
		\end{align}
	We use that $|\phi|\leq ||\phi||$. With the $h$-uniformly equivalence of the norms $||\cdot||_h$ and $||\cdot ||$, for sufficiently small $h$, this yields 
	\begin{align}
	\nonumber m_h(d_h^w,\phi_h^w) &\leq ch^2||\phi_h^w||_h(||{R}_hw||\ +||w||_2\ +||f||_2+||{I}_hf||).
	\end{align}	
	We estimate the defect in the dual norm:
	\begin{align*}
	||d_h^w||_{\ast,h} &= \sup_{\phi_h^w \in V_h\setminus \lbrace 0 \rbrace}\frac{m_h(d_h^w,\phi_h^u)}{||\phi_h^w||_h}\\
	&\leq \sup_{\phi_h^w \in V_h\setminus \lbrace 0 \rbrace}\frac{  ch^2||\phi_h^w||_h(||{R}_hw||\ +||w||_2\ +||f||_2+||{I}_hf||)}{||\phi_h^w||_h}\\
	&\leq ch^2(||{R}_hw||\ +||w||_2\ +||f||_2+||{I}_hf||)
	\end{align*}
	To eliminate $R_hw$ we add and subtract $w$ and use the error bound for the Ritz map (\ref{ritz_v}). To eliminate $I_h f$ we add and subtract $f$ and  use the estimate for the interpolation error in the $||\cdot ||$-norm with (\ref{interpolation_error2}). We obtain
	\begin{align*}
	||d_h^w||_{\ast,h}
	& \leq ch^2(| |R_hw-w||+||w||+||w||_2+||f||_2+||I_hf-f||+||f|| )\\
	& \leq ch^2(h||w||_2+||w||+||w||_2+||f||_2+h||f||_2+||f||  )\\
	& \leq ch^2(||w||_2+||f||_2)
	\end{align*}

	\textit{Part b: Derivative of the defect in $w$.}
	
	We consider the equation for the defect, $d_h^w$, (\ref{def_2}) and differentiate it with respect to $t$ to obtain
	\begin{align*}
	m_h(d_h^w,\phi_h^w)
	&= m_h(\tilde{R}_h\pa_tw,\phi_h^w)-m(\pa_tw,(\phi_h^w)^l)+m(\pa_tf,(\phi_h^w)^l)-m_h(\tilde{I}_h\pa_tf,\phi_h^w).
	\end{align*}
We add and subtract the terms $ m((\tilde{R}_h\partial_tw)^l,(\phi_h^w)^l)$ and $ m(I_h\partial_tf,(\phi_h^w)^l)$, which then it yields
	\begin{align*}
	m_h(\partial_t d_h^w,\phi_h^w) 
	&= (m_h(\tilde{R}_h\partial_tw,\phi_h^w)-m((\tilde{R}_h\partial_tw)^l,(\phi_h^w)^l)+m(R_h\partial_tw-\partial_tw,(\phi_h^w)^l)\\ & \quad+(m(\partial_tf,(\phi_h^w)^l)-m(I_h\partial_tf,(\phi_h^w)^l))+(m(I_h\partial_tf,(\phi_h^w)^l)-m_h(\tilde{I}_h\partial_tf,\phi_h^w))
	\end{align*}
	With that equation and by the steps as for the defect in $w$, we arrive at
	\begin{align*}
	||\pa_t d_h^w(t)||_{\ast,h}
	& \leq ch^2(||\pa_tw||_2+||\pa_tf||_2).
	\end{align*}
	\end{proof}
\section{Convergence}\label{conv}
	This section is only devoted to the proof of Theorem \ref{convergence}. The previously proved statements regarding stability, consistency and control over the Ritz error are now combined to finally give the proof of the theorem.
\begin{proof}[Proof of Theorem \ref{convergence}]

	The proof consists of three parts.
	First the errors  are split up into two parts \textit{(i)}, one part is then bounded by the Ritz map error \textit{(ii)}, on the other part the perturbation result applies, which is then bounded using the consistency statement \textit{(iii)}.\\
	
	\textit{(i) Decomposition of the error}	
	
The errors $u_h^l-u$, $w_h^l-w$ and $\partial_t u_h^l-\partial_t u$ are rewritten as follows
$$u_h^l-u = (u_h-\tilde{R}_h u)^l+(R_hu-u),$$
$$w_h^l-w = (w_h-\tilde{R}_h w)^l+(R_hw-w),$$
$$\partial_t u_h^l-\partial_t u = (\partial_t u_h-\tilde{R}_h \partial_t u)^l+(R_h \partial_t u-\partial_t u).$$

\textit{(ii) Estimate of the Ritz error terms}	

The second terms $R_hu-u, R_hw-w $ and $R_h \partial_t u-\partial_t u $ are estimated with (\ref{ritz_v}), which results in the  following first-order bound
\begin{align*}
||&R_h u(t)-u(t)||^2+||R_h w(t)-w(t)||^2 \\ & +\int_0^t||R_h \partial_tu(s)-\partial_t u(s)||^2\mathrm{d}s+\int_0^t||R_h w(s)-w(s)||^2\mathrm{d}s \\ & \leq c h^2 \Big(||u(t)||_2^2+||w(t)||_2^2+\int_0^t||\partial_t u(s)||_2^2 \mathrm{d}s  +\int_0^t ||w(s)||_2^2\mathrm{d}s \Big)
\end{align*}
and for the $|\cdot |$-norm with  (\ref*{ritz2}) in the second-order error bound
\begin{align*}
|&R_h u(t)-u(t)|^2+|R_h w(t)-w(t)|^2 \\ & +\int_0^t|R_h \partial_tu(s)-\partial_t u(s)|^2\mathrm{d}s+\int_0^t|R_h w(s)-w(s)|^2\mathrm{d}s \\ & \leq c h^4 \Big(|u(t)|_2^2+|w(t)|_2^2+\int_0^t|\partial_t u(s)|_2^2 \mathrm{d}s  +\int_0^t |w(s)|_2^2\mathrm{d}s \Big).
\end{align*}

\textit{(ii) Estimate of the first error terms}

To estimate the first terms we define the errors $e_h^u:=\tilde{R}_hu-u_h$ and $ e_h^w:=\tilde{R}_hw-w_h$.\\
The assumption of $e_h^u(0)=0$ and $e_h^w(0)=0$ implies $d_h^u(0)=0$ and $d_h^w(0)=0$. Applying Proposition \ref{energy_est_semidiscr} we obtain the following stability estimate for these errors
\begin{align*}
||&e^u_h(t)||_h^2+||e^w_h(t)||_h^2+\int_0^t||\partial_te^u_h(s)||_h^2\mathrm{d}s+\int_0^t||e^w_h(s)||_h^2\mathrm{d}s 
 \\  & 
\leq c \Big( ||d^u_h(t)||_{\ast, h}^2+\int_0^t||d^u_h(s)||_{\ast, h}^2 \mathrm{d}s+\int_0^t||\partial_td^u_h(s)||_{\ast, h}^2\mathrm{d}s\\ 
& \qquad +\int_0^t||d^w_h(s)||_{\ast, h}^2\mathrm{d}s+\int_0^t||\partial_td^w_h(s)||_{\ast, h}^2\mathrm{d}s\Big)
\end{align*}
with $c$  depending only on $T$ .\\
Since we chose $u_h^\ast=\tilde{R}_hu$ and $w_h^\ast=\tilde{R}_hw$, the consistency result, Proposition \ref{pr_consist}, applies and we obtain
\begin{align*}
||&e^u_h(t)||_h^2+||e^w_h(t)||_h^2+\int_0^t||\partial_te^u_h(s)||_h^2\mathrm{d}s+\int_0^t||e^w_h(s)||_h^2\mathrm{d}s \\  & 
 \leq ch^4  \Big(||\partial_t u (0)|| ^2+ ||\partial_t u (t)|| ^2 + \int_0^t ||\partial_t u(s)||_2^2 \mathrm{d}s + \int_0^t ||\partial_{tt} u(s)||_2^2 \mathrm{d}s\\
& \quad + \int_0^t ||w(s)||_2^2\mathrm{d}s+\int_0^t ||\partial_t w(s)||_2^2\mathrm{d}s+\int_0^t||f(s)||_2 ^2 \mathrm{d}s+\int_0^t||\partial_t f(s)||_2 ^2 \mathrm{d}s \Big).
\end{align*}
Because it is $|\cdot |\leq||\cdot||$, we obtain as well
\begin{align*}
|&e^u_h(t)|_h^2+|e^w_h(t)|_h^2+\int_0^t|\partial_te^u_h(s)|_h^2\mathrm{d}s+\int_0^t|e^w_h(s)|_h^2\mathrm{d}s \\  & 
\leq ch^4  \Big(||\partial_t u (0)|| ^2+ ||\partial_t u (t)|| ^2 + \int_0^t ||\partial_t u(s)||_2^2 \mathrm{d}s + \int_0^t ||\partial_{tt} u(s)||_2^2 \mathrm{d}s\\
& \quad + \int_0^t ||w(s)||_2^2\mathrm{d}s+\int_0^t ||\partial_t w(s)||_2^2\mathrm{d}s+\int_0^t||f(s)||_2 ^2 \mathrm{d}s+\int_0^t||\partial_t f(s)||_2 ^2 \mathrm{d}s \Big).
\end{align*}
This inequalities, the equivalence of discrete and continuous norms and the bounds for the Ritz error yield	the statements of Theorem \ref{convergence}. 
	\end{proof}
\chapter{Nonlinear case}
The Cahn--Hilliard equation with dynamic Cahn--Hilliard boundary conditions as presented in the introduction included potential functions, which are dependent on $u$. To approach to this situation we will consider the Cahn--Hilliard/Cahn--Hilliard coupling with nonlinearity $f(u)$ in this chapter. We will restrict our function $f$ to have Lipschitz continuity. The formulations and the analysis made before are extended to this more general case. The main part of this chapter will be the proof of a perturbation estimate based on the perturbation result for the linear case. Combined with consistency estimates,  we will then have a convergence theorem with a first- and second-order bound. As in Chapter 1 and 2 the perturbation result proven in this chapter applies not only on the Cahn--Hilliard/Cahn--Hilliard coupling, but to a class of problems that fits into the presented abstract framework, Section \ref{abstract}.

The nonlinear Cahn--Hilliard equation in its second-order formulation is given by
\begin{subequations}\label{non_linear_cahn_bulk}
\begin{align}
\partial_t u &= \Delta w \qquad &\text{in}\  \Omega\\
w &=  -\Delta u + f_\Omega(u) \qquad &\text{in}\  \Omega,
\end{align}
\end{subequations}
with dynamic boundary conditions
\begin{subequations}\label{non_linear_cahn_surface}
\begin{align}
\partial_t u &= \Delta_\Gamma w - \partial_\nu w \qquad &\text{on}\  \Gamma\\
w &=  -\Delta_\Gamma u + f_\Gamma (u) + \partial_\nu u \qquad &\text{on}\  \Gamma.
\end{align}
\end{subequations}
For $f_{\Omega},f_{\Gamma}\in C^1(\mathbb{R}, \mathbb{R})$ we define $f:=(f_\Omega,f_\Gamma)$. We assume that $f$ and $f'$ are Lipschitz continuous with Lipschitz constant $L>0$. $f$ is assumed to be sufficient regular for the following.\\
The nonlinear weak formulation is the same as in the linear case, with the only difference that $f$ depends on $u$ :
\begin{subequations}\label{non_lin_weak}
\begin{align}
m(\partial_t u, \phi^u)+a(w,\phi^u)&=0 \qquad & \forall \phi^u\in V,\\
m(w,\phi^w)-a(u, \phi^w) &= m(f(u),\phi^w)\qquad & \forall \phi^w\in V,
\end{align}
\end{subequations}
where $m(f(u),\phi^w)$ is understood as $m(f(u),\phi^w)=\int_\Omega f_\Omega(u)\phi^w+\int_\Gamma f_\Gamma(\gamma u)\gamma \phi^w$. 

The semidiscrete formulation then reads as: Find  $(u_h,w_h)\in C^1([0,T],V_h)\times L^2(0,T;V_h)$ such that 
\begin{subequations}\label{non_lin_semidis}
\begin{align}
m_h(\partial_t u_h, \phi^u_h)+a_h(w_h,\phi^u_h)&=0 \qquad & \forall \phi^u_h\in V_h,\\
m_h(w_h,\phi^w_h)-a_h(u_h , \phi^w_h) &= m(f(u_h),\phi^w_h) \qquad & \forall \phi^w_h\in V_h.
\end{align}
\end{subequations}
For the convergence theorem we assume that
\begin{align}\label{initial}
u_h(0)=\tilde{R}_h u(0),\qquad w_h(0)=\tilde{R}_h w(0).
\end{align}
An analogous convergence result as in the linear case, Theorem \ref{convergence}, also applies in this nonlinear case:
\begin{theorem}\label{non_lin_convergence}
		Let $(u,w)\in C^2([0,T],\lbrace v\in H^2(\Omega)\cap L^\infty(\Omega) : \gamma v\in H^2(\Gamma)\cap L^\infty(\Gamma)\rbrace )\times C^1([0,T],\lbrace v\in H^2(\Omega) : \gamma v\in H^2(\Gamma)\rbrace)$ be the solution of (\ref{non_linear_cahn_bulk})--(\ref{non_linear_cahn_surface}). Then, for sufficiently small $h\leq h_0$ the semidiscrete solution of (\ref{non_lin_semidis})
		 $(u_h,w_h)\in C^1([0,T],V_h)^2$, with initial values \eqref{initial}, satisfies the following error bounds for all $t$ with $0\leq t\leq T$,
	 \begin{align*}
||&u_h^l(t)-u(t)||^2+||w_h^l(t)-w(t)||^2  \\ \qquad & +\int_0^t||\partial_tu_h^l(s)-\partial_t u(s)||^2\mathrm{d}s +\int_0^t||w_h^l(s)-w(s)||^2\mathrm{d}s  \leq C h^2.
\end{align*}
In the weaker norm, $|\cdot |$, we obtain a second-order error bound
\begin{align*}
|&u_h^l(t)-u(t)|^2+|w_h^l(t)-w(t)|^2  \\ \qquad & +\int_0^t|\partial_tu_h^l(s)-\partial_t u(s)|^2\mathrm{d}s +\int_0^t|w_h^l(s)-w(s)|^2\mathrm{d}s  \leq C h^4.
\end{align*}
with $C>0$ depending on norms of $u,w,f$ and their derivatives and exponentially on $T$, but independent from $h$ and $t$.
\end{theorem}
The proof of this theorem will be a result of a combination of a stability result, consistency and the error bound for the Ritz map. We take one section to discuss stability and one to discuss consistency, before we turn to this proof again.
\section{Stability}
We consider the perturbed equation system with the perturbed solution $(u_h^\ast, w_h^\ast) \in V_h^2$, which is a projected exact solution: 
\begin{subequations}\label{non_per}
	\begin{align*}
	m_h(\partial_t u_h^\ast, \phi^u_h)+a_h(w_h^\ast,\phi^u_h)&=m_h(d_h^u,\phi_h^u) \qquad & \forall \phi^u_h\in V_h\\
	m_h(w_h^\ast,\phi^w)-a_h(u_h^\ast , \phi^w_h) &= m_h(f(u_h^\ast),\phi^w_h)+m_h(d_h^w,\phi_h^w) \qquad & \forall \phi^w_h\in V_h.
	\end{align*}
\end{subequations}
We subtract the exact system \eqref{non_lin_semidis} and define $e_h^u:=u_h^\ast-u_h$ and $ e_h^w:= w_h^\ast-w_h$. Then the error equation system becomes the following:
\begin{subequations}\label{non_linear_defect}
	\begin{align}
	m_h(\partial_t e_h^u, \phi^u_h)+a_h(e^w_h,\phi^u_h)&=m_h(d_h^u,\phi_h^u) \qquad & \forall \phi^u_h\in V_h\\
	m_h(e^w_h,\phi^w)-a_h(e^u_h , \phi^w_h) &= m_h(f(u_h^\ast)-f(u_h),\phi^w_h)+m_h(d_h^w,\phi_h^w) \qquad & \forall \phi^w_h\in V_h.
	\end{align}
\end{subequations}
With stronger assumptions but similar to the linear case the following stability bound holds true: 
\begin{proposition}\label{non_energy_est_semidiscr}
	Let $(e^u_h, e^w_h) \in C^1([0,T];V_h)^2$ be the solution of the error equation system above with defects $d^u_h, d^w_h\in C^1([0,T];V_h)$. Let $u\in C^1([0,T],L^\infty(\Omega))$ and $u_h\in C^1([0,T],V_h)$. Furthermore let the following assumptions hold
	\begin{align}
		&|u-(u_h^\ast)^l|\leq c h^2 \text{ and } |\pa_tu-(\pa_tu_h^\ast)^l|\leq c h^2,\label{cond_ritz}\\
		\ &||d||_{\ast,h} \leq c h^2 \text{ for } d\in\lbrace d_h^u,d_h^w,\pa_td_h^u,\pa_t d_h^w\rbrace,\label{cond_def}\\
		\ &|| e_h^u(0)||_h\leq c h^2  \text{ and }||e_h^w(0)||_h^2\leq c h^2 \label{cond_init},
	\end{align}
	for $c<0$ independent from $h$.

	 Then  for all $t$ with $0<t\leq T$ yields
	\begin{align*}
	||&e_h^u(t)||_h^2+||e_h^w(t)||_h^2+\int_0^t||\partial_te_h^u(s)||_h^2\mathrm{d}s+\int_0^t||e_h^w(s)||_h^2\mathrm{d}s  \\  & 
	\leq C \Big( || e_h^u(0)||_h^2+||e_h^w(0)||_h^2 +|| d_h^u(0)||_{\ast,h}^2+||d_h^u(t)||_{\ast,h}^2\\
	& \qquad+\int_0^t||d_h^u(s)||_{\ast,h}^2 \mathrm{d}s+\int_0^t||\partial_td_h^u(s)||_{\ast,h}^2\mathrm{d}s+\int_0^t||d_h^w(s)||_{\ast,h}^2\mathrm{d}s+\int_0^t||\partial_td_h^w(s)||_{\ast,h}^2\mathrm{d}s\Big),
	\end{align*}
with $C$ depending on $\alpha, \mu,$, exponentially on $T$ and on $u,w,f$ and their derivatives, but independent from $h$ and $t$.
\end{proposition}
The conditions \eqref{cond_ritz}--\eqref{cond_init} will be all satisfied for the semidiscrete weak formulation of the Cahn--Hilliard equation with dynamic Cahn--Hilliard boundary conditions. As the perturbed solution $u_h^\ast$ will be an image of the Ritz map later on, condition \eqref{cond_ritz} is satisfied with \eqref{ritz2}. The second-order bound for the defects \eqref{cond_def} will be shown in Proposition \ref{nl_consis}. The bounds for the initial errors \eqref{cond_init} are satisfied because of our choice for the initial values \eqref{initial}.
\begin{proof}
	For this proof we take intermediate results from the perturbation proof of the linear case, Proposition \ref{energy_est}. We adapt the structure from that proof, by having three sections, a first energy estimate, a second one and then their combination. Since in the error equation there is the additional nonlinear error term, $m_h(f(u_h^\ast)-f(u_h),\phi^w_h)$, the main effort of the proof is to bound its derivative. \\
	Let $0<t^\ast\leq T $ be maximal such that 
	\begin{equation}\label{t_stern}
	|e_h^u(t)|_h\leq h^{\frac{3}{2}
	}\ \forall t\leq t^\ast.
	\end{equation}
	Such an $t^\ast$ exists, because $|e_h^u(0)|_h\leq c h^2$.
	The following is shown for $t\leq t^\ast$.\\
	\textit{(i) First Energy Estimate}\\
	We adapt the estimate (\ref{ee1}), change it into a semidiscrete version and complement it by the nonlinear term of the defect in $w$ to obtain
	\begin{align*}
	||e_h^u(t)||_h^2+\int_0^t ||e_h^w(s)||_h^2\mathrm{d}s   & \leq c\Big(||e_h^u(0)||_h^2+\nonumber
	\int_0^t||d_h^u(s)||_{\ast, h}^2\mathrm{d}s +\int_0^t||d_h^w(s)||_{\ast, h}^2\mathrm{d}s\\
	& \qquad +\int_0^t||f(u_h^\ast(s))-f(u_h(s))||_{\ast, h}^2\mathrm{d}s \Big)+\frac{1}{2}\int_0^t||\partial_t e_h^u(s)||_h^2\mathrm{d}s.
	\end{align*}
	The nonlinear defect term can be estimated, using the Lipschitz continuity of $f$, as follows
	\begin{align*}
	||f(u_h^\ast
	)-f(u_h)||_{\ast, h} &\leq c |f(u_h^\ast)-f(u_h)|_{h}
	\leq c L |u_h^\ast-u_h|_{h}
	\leq c |e_h^u|_h.
	\end{align*}
	We insert this in the above estimate and obtain
	\begin{align}\label{nee1}
	||e_h^u(t)||_h^2+\int_0^t ||e_h^w(s)||_h^2\mathrm{d}s   & \leq c\Big(||e_h^u(0)||_h^2+\int_0^t|e_h^u|_h^2\mathrm{d}s\nonumber
	+\int_0^t||d_h^u(s)||_{\ast, h}^2\mathrm{d}s \\
	& \qquad  +\int_0^t||d_h^w(s)||_{\ast, h}^2\mathrm{d}s\Big)+\frac{1}{2}\int_0^t||\partial_t e_h^u(s)||_h^2\mathrm{d}s.
	\end{align}
	\textit{(ii) Second energy estimate.}\\
	This section is structured in five subsections. Part a, b and c consider an energy estimate in the $|\cdot|_h$-norm, Part d an estimate in the $ ||\cdot||_{a,h}$-norm and Part e combines the estimates. Part b is an parenthesis for estimating particular terms and Part c sequels Part a.\\
	\textit{Part a:}\\
	The semidiscrete variant of (\ref{ee2an}) with nonlinear defect term is the following
	\begin{align*}
	|e_h^w(t)|_h^2+\int_0^t|\pa_t e_h^u(s)|_h^2\mathrm{d}s &\leq |e_h^w(0)|_h^2+\eps\int_0^t||\pa_t e_h^u(s)||_h^2\mathrm{d}s + \frac{1}{\eps}\int_0^t||d_h^u(s)||_{\ast, h}^2\mathrm{d}s\\
	& \qquad +2\int_0^tm_h(\pa_td_h^w(s),e_h^w(s))\mathrm{d}s\\
	& \qquad +2\int_0^tm_h(\pa_t(f(u_h^\ast(s))-f(u_h(s))),e_h^w(s))\mathrm{d}s.
	\end{align*}
Estimating the linear term of the derivative of the defect with (\ref{dual}) and Young's inequality yields
	\begin{align}\label{non_lin_z2}
	|e_h^w(t)|_h^2+\int_0^t|\pa_t e_h^u(s)|_h^2\mathrm{d}s &\leq |e_h^w(0)|_h^2+\eps\int_0^t||\pa_t e_h^u(s)||_h^2\mathrm{d}s+\int_0^t|| e_h^w(s)||_h^2\mathrm{d}s \nonumber\\
	& \qquad+ \frac{1}{\eps}\int_0^t||d_h^u(s)||_{\ast, h}^2\mathrm{d}s +\int_0^t||\pa_t d_h^w(s)||_{\ast, h}^2\mathrm{d}s\nonumber\\
	& \qquad +2\int_0^tm_h(\pa_t(f(u_h^\ast(s))-f(u_h(s))),e_h^w(s))\mathrm{d}s.
	\end{align}
We estimate the time derivative of the nonlinear term. By applying the chain rule we obtain
	\begin{align*}
	m_h(\pa_t(f(u_h^\ast)-f(u_h)),e_h^w) & = m_h(f'(u_h^\ast)\pa_t u_h^\ast -f'(u_h)\pa_t u_h,e_h^w)\\
	& = m_h(f'(u_h^\ast)\pa_t u_h^\ast -f'(u_h)\pa_t u_h^\ast ,e_h^w)\\
	& \qquad +m_h(f'(u_h)\pa_t u_h^\ast -f'(u_h)\pa_t u_h,e_h^w)\\
	& = m_h((f'(u_h^\ast) -f'(u_h))\pa_t u_h^\ast ,e_h^w)+m_h(f'(u_h)(\pa_t u_h^\ast -\pa_t u_h),e_h^w)
	\end{align*}
	The Cauchy-Schwarz inequality and the splitting of the factors into its bulk and boundary component then yields
	\begin{align*}
	m_h(\pa_t(f(u_h^\ast)-f(u_h)),e_h^w) 
	&\leq  |(f'(u_h^\ast) -f'(u_h))\pa_t u_h^\ast|_h| e_h^w|_h+|f'(u_h)(\pa_t u_h^\ast -\pa_t u_h)|_h|e_h^w|_h\\
	& = ||(f_\Omega'(u_h^\ast) -f_\Omega'(u_h))\pa_t u_h^\ast||_{L^2(\Omega_h)}| e_h^w|_h\\
	& \qquad +||f_\Omega '(u_h)(\pa_t u_h^\ast -\pa_t u_h)||_{L^2(\Omega_h)}|e_h^w|_h\\
	&\qquad  + ||(f_\Gamma'(\gamma_h u_h^\ast) -f_\Gamma'(\gamma_h u_h))\gamma_h(\pa_t u_h^\ast)||_{L^2(\Gamma_h)}| e_h^w|_h\\
	& \qquad +||f_\Gamma'(\gamma_h u_h)\gamma_h(\pa_t u_h^\ast -\pa_t u_h)||_{L^2(\Gamma_h)}|e_h^w|_h
	\end{align*}
	 We apply the generalized H\"older inequality to the first factors of all for terms and obtain
	 	\begin{align*}
	 	m_h(\pa_t(f(u_h^\ast)-f(u_h)),e_h^w) 
	 	&\leq   ||f_\Omega'(u_h^\ast) -f_\Omega'(u_h)||_{L^2(\Omega_h)}||\pa_t u_h^\ast||_{L^\infty(\Omega_h)}| e_h^w|_h\\
	 	& \qquad +||f_\Omega'(u_h)||_{L^\infty(\Omega_h)}||\pa_t u_h^\ast -\pa_t u_h||_{L^2(\Omega_h)}|e_h^w|_h\\
	 	&\qquad  + ||f_\Gamma'(\gamma_h u_h^\ast) -f_\Gamma'(\gamma_h u_h)||_{L^2(\Gamma_h)}||\pa_t u_h^\ast||_{L^\infty(\Gamma_h)}| e_h^w|_h\\
	 	& \qquad +||f_\Gamma'(\gamma_h u_h)||_{L^\infty(\Gamma_h)}||\gamma_h(\pa_t u_h^\ast -\pa_t u_h)||_{L^2(\Gamma_h)}|e_h^w|_h
	 	\end{align*}
	We use the Lipschitz continuity of $f'$, which then yields
	\begin{align}\label{non_lin_z1}
	m_h(\pa_t(f(u_h^\ast)-f(u_h)),e_h^w) 
	&\leq L ||e_h^u||_{L^2(\Omega_h)}||\pa_t u_h^\ast||_{L^\infty(\Omega_h)}| e_h^w|_h\nonumber \\
	& \qquad +||f_\Omega'(u_h)||_{L^\infty(\Omega_h)}||\pa_t e_h^u||_{L^2(\Omega_h)} |e_h^w|_h\nonumber \\
	&\qquad  + L||\gamma_h e_h^u||_{L^2(\Gamma_h)}||\gamma_h(\pa_t u_h^\ast)||_{L^\infty(\Gamma_h)}| e_h^w|_h\nonumber\\
	& \qquad +||f_\Gamma'(\gamma_h u_h)||_{L^\infty(\Gamma_h)}||\gamma_h(\pa_t e_h^u)||_{L^2(\Gamma_h)}|e_h^w|_h
	\end{align}
		\textit{Part b: Estimation of }$||\pa_t u_h^\ast||_{L^\infty(\Omega_h)}, ||\gamma_h(\pa_t u_h^\ast)||_{L^\infty(\Gamma_h)} $\textit{ and} $||f_\Omega'(u_h)||_{L^\infty(\Omega_h)}$, $||f_\Gamma'(\gamma_h u_h)||_{L^\infty(\Gamma_h)} .$\\
		To bound these terms we will apply the inverse estimate as shown in \textit{Brenner\&Scott (2008)}, \cite[Theorem ~4.5.11]{brenner}. Beside the equivalence of continuous and discrete norm as shown in Lemma \ref{equivalence} we need an equivalence in the $L^\infty$-norms. The discrete and continuous maximums norms are even equal, because the lift operator is a composition with a bijective function.
		
		I. \textit{Estimation of} $||\pa_t u_h^\ast||_{L^\infty(\Omega_h)}$
		
		We insert $\tilde{I}_h\pa_tu$ and use the triangle inequality to obtain
	\begin{align*}
	||\pa_t u_h^\ast||_{L^\infty(\Omega_h)} & \leq ||\pa_t u_h^\ast- \tilde{I}_h\pa_tu||_{L^\infty(\Omega_h)}+||\tilde{I}_h\pa_tu||_{L^\infty(\Omega_h)}.
	\end{align*}
	The inverse estimate yields
	\begin{align*}
	||\pa_t u_h^\ast||_{L^\infty(\Omega_h)} & \leq c h^{-\frac{d}{2}}||\pa_t u_h^\ast- \tilde{I}_h\pa_tu||_{L^2(\Omega_h)}+||\tilde{I}_h\pa_tu||_{L^\infty(\Omega_h)}.
	\end{align*}
	With the equivalence of norms it is obtained
	\begin{align*}
	||\pa_t u_h^\ast||_{L^\infty(\Omega_h)} & \leq c (h^{-\frac{d}{2}}||(\pa_t u_h^\ast)^l- {I}_h\pa_tu||_{L^2(\Omega)}+||{I}_h\pa_tu||_{L^\infty(\Omega)}).
	\end{align*}
	We add and subtract $\pa_tu$ and use the triangle inequality, which yields
	\begin{align*}
	||\pa_t u_h^\ast||_{L^\infty(\Omega_h)} & \leq c (h^{-\frac{d}{2}}(||(\pa_t u_h^\ast)^l-\pa_tu||_{L^2(\Omega)}+ ||\pa_tu- {I}_h\pa_tu||_{L^2(\Omega)})\\ & \qquad+||{I}_h\pa_tu-\pa_tu||_{L^\infty(\Omega)}+||\pa_tu||_{L^\infty(\Omega)}).
	\end{align*}
	Prerequisite \eqref{cond_ritz} and the interpolation error \eqref{interpolation_bulk} bound yields
	\begin{align*}
	||\pa_t u_h^\ast||_{L^\infty(\Omega_h)} & \leq c h^{-\frac{d}{2}}(h^2+ h^2 ||\pa_tu||_{H^2(\Omega)})+||{I}_h\pa_tu-\pa_tu||_{L^\infty(\Omega)}+||\pa_tu||_{L^\infty(\Omega)}.
	\end{align*}
	The second to last term is estimated with \eqref{inter_infty_bulk} and we obtain
	\begin{align*}
	||\pa_t u_h^\ast||_{L^\infty(\Omega_h)} & \leq c h^{-\frac{d}{2}}(h^2+ h^2 ||\pa_tu||_{H^2(\Omega)})+c||\pa_tu||_{L^\infty(\Omega)}+||\pa_tu||_{L^\infty(\Omega)}.
	\end{align*}
	For $d\leq 4$ we arrive at $$||\pa_t u_h^\ast||_{L^\infty(\Omega_h)}  \leq c.$$
	
	II. \textit{Estimation of} $||f_\Omega'(u_h)||_{L^\infty(\Omega_h)}$
	
	We add and subtract $f_\Omega'(u_h^\ast)$ and obtain
	\begin{align*}
	||f_\Omega'(u_h)||_{L^\infty(\Omega_h)} & \leq ||f_\Omega'(u_h)-f_\Omega'(u_h^\ast)||_{L^\infty(\Omega_h)} + ||f_\Omega'(u_h^\ast)||_{L^\infty(\Omega_h)}.
	\end{align*}
		We use the norm equivalence and add and subtract $f_\Omega'(u)$ to obtain
		\begin{align*}
		||f_\Omega'(u_h)||_{L^\infty(\Omega_h)} & \leq c (||f_\Omega'(u_h)-f_\Omega'(u_h^\ast)||_{L^\infty(\Omega_h)} + ||f_\Omega'((u_h^\ast)^l)-f_\Omega'(u)||_{L^\infty(\Omega)}+ ||f_\Omega'(u)||_{L^\infty(\Omega)}).
		\end{align*}
		The application of Lipschitz continuity yields
		\begin{align*}
		||f_\Omega'(u_h)||_{L^\infty(\Omega_h)} & \leq c(L||u_h-u_h^\ast||_{L^\infty(\Omega_h)} + L||(u_h^\ast)^l-u||_{L^\infty(\Omega)}+ ||f_\Omega'(u)||_{L^\infty(\Omega)}).
		\end{align*}
			We add and subtract $I_hu$ and obtain
			\begin{align*}
			||f_\Omega'(u_h)||_{L^\infty(\Omega_h)} & \leq c(||u_h-u_h^\ast||_{L^\infty(\Omega_h)} + ||(u_h^\ast)^l-I_hu||_{L^\infty(\Omega)}\\ & \qquad  + ||u-I_hu||_{L^\infty(\Omega)}+ ||f_\Omega'(u)||_{L^\infty(\Omega)}).
			\end{align*}
	We use the inverse estimate and obtain
	\begin{align*}
	||f_\Omega'(u_h)||_{L^\infty(\Omega_h)} & \leq c(h^{-\frac{d}{2}}||u_h-u_h^\ast||_{L^2(\Omega_h)} + h^{-\frac{d}{2}}||u_h^\ast-\tilde{I}_hu||_{L^2(\Omega_h)}\\ & \qquad  + ||u-I_hu||_{L^\infty(\Omega)}+ ||f_\Omega'(u)||_{L^\infty(\Omega)}).
	\end{align*}
	The addition and subtraction of $u$ yields
	\begin{align*}
	||f_\Omega'(u_h)||_{L^\infty(\Omega_h)} & \leq c(h^{-\frac{d}{2}}||u_h-u_h^\ast||_{L^2(\Omega_h)} + h^{-\frac{d}{2}}(||(u_h^\ast)^l-u||_{L^2(\Omega)}+||I_hu-u||_{L^2(\Omega)})\\ & \qquad  + ||u-I_hu||_{L^\infty(\Omega)}+ ||f_\Omega'(u)||_{L^\infty(\Omega)}).
	\end{align*}
	On the second last term we apply the interpolation error (\ref{inter_infty_bulk}) and obtain
	\begin{align*}
	||f_\Omega'(u_h)||_{L^\infty(\Omega_h)} & \leq c(h^{-\frac{d}{2}}||e_h^u||_{L^2(\Omega_h)} + h^{-\frac{d}{2}}||(u_h^\ast)^l-u||_{L^2(\Omega)}+h^{-\frac{d}{2}}||I_hu-u||_{L^2(\Omega)}\\ & \qquad  + ||u||_{L^\infty(\Omega)}+ ||f_\Omega'(u)||_{L^\infty(\Omega)}).
	\end{align*}
	
	We obtain second-order bounds  for $||(u_h^\ast)^l-u||_{L^2(\Omega)} $ with (\ref{cond_ritz}) and for $||I_hu-u||_{L^2(\Omega)} $ with the interpolation error (\ref{interpolation_bulk}). The term $||e_h^u||_{L^2(\Omega_h)}$ is bounded with (\ref{t_stern}), hence, for $d\leq 3$, $||f_\Omega'(u_h)||_{L^\infty(\Omega_h)}$ can be bounded independently from $h$.
	
	III. \textit{Estimation of} $||\gamma_h\pa_t u_h^\ast||_{L^\infty(\Gamma_h)}\  and\  ||f_\Gamma'(\gamma_h u_h)||_{L^\infty(\Gamma_h)}$
		
		The estimates as done in I. and II. can be transferred to the surface case. The only difference is the exponent of $h$ in the inverse estimate, which is $-\frac{d-1}{2}$, because the finite elements are one dimension lower. \\	
		\textit{Part c. Sequel of Part a: }\\
		The following is shown for $t\leq t^\ast$. We take (\ref{non_lin_z1}) and use that $||\pa_t u_h^\ast||_{L^\infty(\Omega_h)}$, $||\pa_t u_h^\ast||_{L^\infty(\Gamma_h)} $ and $||f'(u_h)||_{L^\infty(\Omega_h)}, ||f'(u_h)||_{L^\infty(\Gamma_h)} $ are bounded independently of $h$ and obtain	
		\begin{align*}
		m_h(\pa_t(f(u_h^\ast)-f(u_h)),e_h^w) 
		& \leq c (L|e_h^u|_h\  |e_h^w|_h +|\pa_t e_h^u|_h\ |e_h^w|_h).
		\end{align*}
		We use the Young inequality, the $\eps$-Young inequality, which then yields
			\begin{align*}
			m_h(\pa_t(f(u_h^\ast)-f(u_h)),e_h^w) 
			& \leq c(|e_h^u|_h^2+ \eps'|\pa_t e_h^u|_h^2+ |e_h^w|_h^2 +\frac{1}{\eps'} |e_h^w|_h^2).
			\end{align*}
		This estimate for $m_h(\pa_t(f(u_h^\ast)-f(u_h)),e_h^w) $ is plugged into (\ref{non_lin_z2}) and we arrive at
			\begin{align*}
			|e_h^w(t)|_h^2 +\int_0^t|\pa_t e_h^u(s)|_h^2\mathrm{d}s&\leq c\Big(|e_h^w(0)|_h^2 +\int_0^t|e_h^u(s)|_h^2\mathrm{d}s+\eps'\int_0^t |\pa_t e_h^u(s)|_h^2\mathrm{d}s\\ &\qquad +\eps\int_0^t||\pa_t e_h^u(s)||_h^2\mathrm{d}s +\Big(1+\frac{1}{\eps'} \Big)\int_0^t|e_h^w(s)|_h^2\mathrm{d}s\\ &\qquad+\int_0^t|| e_h^w(s)||_h^2\mathrm{d}s+ \frac{1}{\eps}\int_0^t||d_h^u(s)||_{\ast, h}^2\mathrm{d}s+\int_0^t||\pa_t d_h^w(s)||_{\ast, h}^2\mathrm{d}s\Big).
			\end{align*}	
			We choose $\eps':=\frac{1}{2c}$ that the third term on the left-hand side is absorbed, and obtain
				\begin{align}\label{non_lin_z4}
				|e_h^w(t)|_h^2 +\int_0^t|\pa_t e_h^u(s)|_h^2\mathrm{d}s&\leq c\Big(|e_h^w(0)|_h^2 +\int_0^t|e_h^u(s)|_h^2\mathrm{d}s +\eps\int_0^t||\pa_t e_h^u(s)||_h^2\mathrm{d}s +\int_0^t|e_h^w(s)|_h^2\mathrm{d}s\nonumber\\ &\qquad+\int_0^t|| e_h^w(s)||_h^2\mathrm{d}s+ \frac{1}{\eps}\int_0^t||d_h^u(s)||_{\ast, h}^2\mathrm{d}s+\int_0^t||\pa_t d_h^w(s)||_{\ast, h}^2\mathrm{d}s\Big).
				\end{align}	
		\textit{Part d:}\\
		We take the inequality (\ref{ee2bn}) with semidiscrete notation and a nonlinear defect:
		\begin{align*}
		||e_h^w(t)||_{a,h}^2+\int_0^t||\pa_t e_h^u(s)||_{a,h}^2\mathrm{d}s&\leq ||e_h^w(0)||_{a,h}^2 +||e_h^w(0)||_{h}^2+\eps '||e^w(t)||_{h}^2+ ||d_h^u(0)||_{\ast,h}^2\\
		& \qquad +\frac{1}{\eps '}||d_h^u(t)||_{\ast,h}^2+\int_0^t||e_h^w(s)||_h^2\mathrm{d}s+ \int_0^t||d_h^u(s)||_{\ast, h}^2\mathrm{d}s\\
		& \qquad +2\Big|\int_0^tm_h(\pa_td_h^w(s),\pa_t e_h^u(s))\mathrm{d}s\Big|\\
		& \qquad +2\Big|\int_0^tm_h(\pa_t(f(u_h^\ast(s))-f(u_h(s))),\pa_t e_h^u(s))\mathrm{d}s\Big|.
		\end{align*}
		For the second to last term the estimate (\ref{dual}) and the $\eps$-Young inequality are applied, and yield
		\begin{align}\label{non_lin_z3}
		||e_h^w(t)||_{a,h}^2+\int_0^t||\pa_t e_h^u(s)||_{a,h}^2\mathrm{d}s &\leq ||e_h^w(0)||_{a,h}^2 +||e_h^w(0)||_{h}^2+\eps '||e^w(t)||_{h}^2+ ||d_h^u(0)||_{\ast,h}^2\nonumber\\
		& \qquad +\frac{1}{\eps '}||d_h^u(t)||_{\ast,h}^2+\int_0^t||e_h^w(s)||_h^2\mathrm{d}s+\eps\int_0^t||\pa_t e_h^u(s)||_h^2\mathrm{d}s\nonumber\\
		& \qquad + \int_0^t||d_h^u(s)||_{\ast, h}^2\mathrm{d}s+\frac{1}{\eps}\int_0^t||\pa_t d_h^w(s)||_{\ast, h}^2\mathrm{d}s\nonumber\\
		& \qquad +2\Big|\int_0^tm_h(\pa_t(f(u_h^\ast(s))-f(u_h(s))),\pa_t e_h^u(s))\mathrm{d}s\Big|.
		\end{align}
	
		We consider the integrand of the last term. With the same proof as in Part a and b we obtain
		\begin{align*}
		m_h(\pa_t(f(u_h^\ast)-f(u_h)),\pa_t e_h^u) 
		& \leq c (|e_h^u|_h\  |\pa_t e_h^u|_h + |\pa_t e_h^u|_h\ |\pa_t e_h^u|_h).
		\end{align*}
		The Young inequality yields
		\begin{align*}
		m_h(\pa_t(f(u_h^\ast)-f(u_h)),\pa_t e_h^u) 
		& \leq c_0(|e_h^u|_h^2  +|\pa_t e_h^u|_h^2).
		\end{align*}
		This estimate is plugged into (\ref{non_lin_z3}), and then we obtain
		\begin{align}\label{nl_z10}
		||e_h^w(t)||_{a,h}^2+\int_0^t||\pa_t e_h^u(s)||_{a,h}^2\mathrm{d}s &\leq ||e_h^w(0)||_{a,h}^2 +||e_h^w(0)||_{h}^2+\eps '||e^w(t)||_{h}^2+ ||d_h^u(0)||_{\ast,h}^2\nonumber\\
		& \qquad +\frac{1}{\eps '}||d_h^u(t)||_{\ast,h}^2+c_0\int_0^t|e_h^u(s)|_h^2\mathrm{d}s +c_0\int_0^t|\pa_t e_h^u(s)|_h^2\mathrm{d}s\nonumber\\
		& \qquad+\int_0^t||e_h^w(s)||_h^2\mathrm{d}s+\eps\int_0^t||\pa_t e_h^u(s)||_h^2\mathrm{d}s\nonumber\\
		& \qquad + \int_0^t||d_h^u(s)||_{\ast, h}^2\mathrm{d}s+\frac{1}{\eps}\int_0^t||\pa_t d_h^w(s)||_{\ast, h}^2\mathrm{d}s.
		\end{align}
		\textit{Part e: Adding the Estimates:}\\
		We multiply (\ref{non_lin_z4}) by $2c_0$ and add it to the estimate \eqref{nl_z10}, which then yields
	\begin{align*}
	2c_0&|e_h^w(t)|_h^2+||e_h^w(t)||_{a,h}^2+2c_0\int_0^t|\pa_t e_h^u(s)|_h^2\mathrm{d}s+ \int_0^t||\pa_t e_h^u(s)||_{a,h}^2\mathrm{d}s\\ &\leq 2c_0c\Big(|e_h^w(0)|_h^2 +\int_0^t|e_h^u(s)|_h^2\mathrm{d}s +\eps\int_0^t||\pa_t e_h^u(s)||_h^2\mathrm{d}s \nonumber+\int_0^t|| e_h^w(s)||_h^2\mathrm{d}s\\ &\qquad+ \frac{1}{\eps}\int_0^t||d_h^u(s)||_{\ast, h}^2\mathrm{d}s+\int_0^t||\pa_t d_h^w(s)||_{\ast, h}^2\mathrm{d}s\Big)+||e_h^w(0)||_{a,h}^2 +||e_h^w(0)||_{h}^2\\ &\qquad+\eps '||e^w(t)||_{h}^2+ ||d_h^u(0)||_{\ast,h}^2\nonumber+\frac{1}{\eps '}||d_h^u(t)||_{\ast,h}^2+c_0\int_0^t|e_h^u(s)|_h^2\mathrm{d}s +c_0\int_0^t|\pa_t e_h^u(s)|_h^2\mathrm{d}s\\
	& \qquad+\int_0^t||e_h^w(s)||_h^2\mathrm{d}s+\eps\int_0^t||\pa_t e_h^u(s)||_h^2\mathrm{d}s\nonumber+ \int_0^t||d_h^u(s)||_{\ast, h}^2\mathrm{d}s+\frac{1}{\eps}\int_0^t||\pa_t d_h^w(s)||_{\ast, h}^2\mathrm{d}s.
	\end{align*}
	The integral term with $|\pa_t e_h^u(s)|_h^2$ is absorbed and we obtain
	\begin{align*}
	||e_h^w(t)||_h^2 +\int_0^t||\pa_t e_h^u(s)||_h^2\mathrm{d}s&\leq c\Big(||e_h^w(0)||_h^2 +\eps ||e^w(t)||_{h}^2 + ||d_h^u(0)||_{\ast,h}^2+\frac{1}{\eps }||d_h^u(t)||_{\ast,h}^2 \\
	& \qquad +\int_0^t|e_h^u(s)|_h^2\mathrm{d}s+\eps\int_0^t||\pa_t e_h^u(s)||_h^2\mathrm{d}s+\int_0^t|| e_h^w(s)||_h^2\mathrm{d}s\\
	& \qquad+ \Big(1+\frac{1}{\eps}\Big)\int_0^t||d_h^u(s)||_{\ast, h}^2\mathrm{d}s+\Big(1+\frac{1}{\eps}\Big)\int_0^t||\pa_t d_h^w(s)||_{\ast, h}^2\mathrm{d}s\Big).
	\end{align*}
	To absorb the  on the second and sixth term on the right-hand side we choose $\eps :=\frac{1}{2c}$ which yields
			\begin{align*}
			||e_h^w(t)||_h^2 +\int_0^t||\pa_t e_h^u(s)||_h^2\mathrm{d}s&\leq c\Big(||e_h^w(0)||_h^2 + ||d_h^u(0)||_{\ast,h}^2+||d_h^u(t)||_{\ast,h}^2  +\int_0^t|e_h^u(s)|_h^2\mathrm{d}s\\
			& \qquad+\int_0^t|| e_h^w(s)||_h^2\mathrm{d}s+ \int_0^t||d_h^u(s)||_{\ast, h}^2\mathrm{d}s+\int_0^t||\pa_t d_h^w(s)||_{\ast, h}^2\mathrm{d}s\Big).
			\end{align*}
		To lose the $ || e_h^w(s)||_h^2$-integral term we applied Gronwall's inequality and obtain
		\begin{align}\label{nee2}
		||e_h^w(t)||_h^2 +\int_0^t||\pa_t e_h^u(s)||_h^2\mathrm{d}s&\leq c\Big(||e_h^w(0)||_h^2 + ||d_h^u(0)||_{\ast,h}^2+||d_h^u(t)||_{\ast,h}^2  +\int_0^t|e_h^u(s)|_h^2\mathrm{d}s\nonumber\\
		& \qquad + \int_0^t||d_h^u(s)||_{\ast, h}^2\mathrm{d}s+\int_0^t||\pa_t d_h^w(s)||_{\ast, h}^2\mathrm{d}s\Big)
		\end{align}
		\textit{(iii) Combination of Energy Estimates}\\
		We add (\ref{nee1}) and (\ref{nee2}) and obtain for $t\leq t^\ast$
			\begin{align*}
			||&e_h^u(t)||_h^2+||e_h^w(t)||_h^2 +\int_0^t||\pa_t e_h^u(s)||_h^2\mathrm{d}s+\int_0^t ||e_h^w(s)||_h^2\mathrm{d}s\\&\leq c\Big(||e_h^u(0)||_h^2  +||e_h^w(0)||_h^2+ ||d_h^u(0)||_{\ast,h}^2+||d_h^u(t)||_{\ast,h}^2 +\int_0^t|e_h^u(s)|_h^2\mathrm{d}s\\
			& \qquad + \int_0^t||d_h^u(s)||_{\ast, h}^2\mathrm{d}s+\int_0^t||d_h^w(s)||_{\ast, h}^2\mathrm{d}s+\int_0^t||\pa_t d_h^w(s)||_{\ast, h}^2\mathrm{d}s  \Big)+\frac{1}{2}\int_0^t||\partial_t e_h^u(s)||^2\mathrm{d}s.
			\end{align*}
			The last term is absorbed, which yields
			 	\begin{align}\label{non_lin_z5}
			 	||&e_h^u(t)||_h^2+||e_h^w(t)||_h^2 +\int_0^t||\pa_t e_h^u(s)||_h^2\mathrm{d}s+\int_0^t ||e_h^w(s)||_h^2\mathrm{d}s\nonumber\\&\leq c\Big(||e_h^u(0)||_h^2  +||e_h^w(0)||_h^2+ ||d_h^u(0)||_{\ast,h}^2+||d_h^u(t)||_{\ast,h}^2 +\int_0^t|e_h^u(s)|_h^2\mathrm{d}s\nonumber\\
			 	& \qquad + \int_0^t||d_h^u(s)||_{\ast, h}^2\mathrm{d}s+\int_0^t||d_h^w(s)||_{\ast, h}^2\mathrm{d}s+\int_0^t||\pa_t d_h^w(s)||_{\ast, h}^2\mathrm{d}s  \Big).
			 	\end{align}
			 	
			Now it remains to show that $t^\ast =T$.\\
			Using prerequisites \eqref{cond_def} and \eqref{cond_init} we know that the right-hand side of (\ref{non_lin_z5}) is $\mathcal{O}(h^4)$, this yields
			\begin{align*}
			|e_h^u(t)|_h^2 \leq||e_h^u(t)||_h^2 \leq c h^4 \qquad \text{for } t\leq t^\ast
			\end{align*}
			For $h^\frac{1}{2}\leq \frac{c}{2}$ this implies
			$$|e_h^u(t)|_h \leq \frac{h^{\frac{3}{2}}}{2} \qquad \text{for } t\leq t^\ast,$$
			because of continuity this contradicts the maximality of $t^\ast$, the only option is $t^\ast=T$.
	\end{proof}

\section{Consistency}
In this section we prove error estimates for the defects and their derivatives, therefore also proving condition (\ref{cond_def}) of Proposition \ref{non_energy_est_semidiscr}.

\begin{proposition}\label{nl_consis}
	Let $(u,w)$ be a solution of (\ref{non_linear_cahn_bulk})--(\ref{non_linear_cahn_surface}) that satisfies the regularity assumptions of Theorem \ref{non_lin_convergence}. We define the perturbed solution $u_h^\ast:=R_hu$, $w_h^\ast:= R_hw$ using the Ritz map.
	Then there exist defects $d_h^u, d_h^w\in C^1([0,t];V_h)$ such that 
	$(u_h^\ast,w_h^\ast)$ is a solution of the system (\ref{non_per}).
	Then these defects satisfy
	\begin{align}
		||d_h^u(t)||_{\ast,h}^2 \leq C h^4  &\text{ and } ||\pa_t d_h^u(t)||_{\ast,h}^2 \leq C h^4, \\
	||d_h^w(t)||_{\ast,h}^2 \leq C h^4  &\text{ and } ||\pa_t d_h^w(t)||_{\ast,h}^2 \leq C h^4,
	\end{align}
	with $C>0$ depending on $T, L$ and on norms of $u,w,f$ and their derivatives, but independent from $h$ and $t$.
\end{proposition}
\begin{proof}
	The existence for the defects can be obtained analogous to the linear case, cf. proof of Proposition \ref{pr_consist}.
	
	This proof consists of three parts, the first part treats the defect in $u$ and its derivative, the second part deals with the estimate of the defect in $w$ and third with its derivative. \\
\textit{(i) Defect in $u$ and its derivative}\\
The equations for the defect in $u$, $d_h^u$, and its derivative stay the same as in the linear case (\ref{defect_1}):
\begin{equation*}
m(d_h^u,\phi_h^u)=(m_h(\tilde{R}_h\partial_t u,\phi_h^u)-m((\tilde{R}_h \partial_t u)^l,(\phi_h^u)^l))-m(\partial_t u-R_h \partial_tu,(\phi_h^u)^l),
\end{equation*}
and (\ref{der_def_1})
\begin{equation*}
m_h(\partial_t d_h^u,\phi_h^u) = (m_h(\tilde{R}_h\partial_{tt} u,\phi_h^u)-m((\tilde{R}_h \partial_{tt} u)^l,(\phi_h^u)^l))-m(\partial_{tt} u-R_h \partial_{tt}u,(\phi_h^u)^l),
\end{equation*}
which yields the same bounds then on $d_h^u$ and $\pa_t d_h^u$, cf. proof of Proposition  \ref{pr_consist}.\\
\textit{(ii) Defect in $w$}\\
For the defect in $w$ we consider the error equations \eqref{non_linear_defect}. Analogues to (\ref{defect_2}) in the linear case, we obtain, by complement the dependency of $f$ from $u$ and add the nonlinear error term
\begin{align*}
m_h(d_h^w,\phi_h^w)+m_h(f(\tilde{R}_h u)-f(u_h),\phi^w_h) &= m_h(\tilde{R}_hw,\phi_h^w)-m((\tilde{R}_hw)^l,(\phi_h^w)^l)\\ & \qquad+m(R_hw-w,(\phi_h^w)^l)\\ & \qquad+m(f(u),(\phi_h^w)^l)-m_h(f(u_h),\phi_h^w).
\end{align*}
Bringing the second term on the left-hand side to the right-hand side yields
\begin{align}\label{non_def}
m_h(d_h^w,\phi_h^w)&= m_h(\tilde{R}_hw,\phi_h^w)-m((\tilde{R}_hw)^l,(\phi_h^w)^l)\nonumber\\ & \qquad+m(R_hw-w,(\phi_h^w)^l)\nonumber\\ & \qquad+m(f(u),(\phi_h^w)^l)-m_h(f(\tilde{R}_h u),\phi_h^w).
\end{align}
The first two terms on the right-hand side can be estimated as in the linear case. We now consider the last term. Because $f(v_h^l)=f(v_h\circ G_h^{-1})=(f\circ v_h)\circ G_h^{-1}=(f(v_h))^l\ \forall v_h\in V_h$ we obtain
\begin{align*}
m(f(u),(\phi_h^w)^l)-m_h(f(\tilde{R}_h u),\phi_h^w) & = m(f(u),(\phi_h^w)^l)-m(f({R}_h u),\phi_h^w)\\ 
& \qquad +m((f(\tilde{R}_h u))^l,(\phi_h^w)^l)-m_h(f(\tilde{R}_h u),\phi_h^w) .
\end{align*}
The first term on the right-hand side can be estimated with the Lipschitz continuity and the error for the Ritz map (\ref{ritz2}). For the second term on the right-hand side we apply the geometric approximation error (\ref{geom_appr_error}), which then yields
\begin{align*}
m((f(\tilde{R}_h u))^l,(\phi_h^w)^l)-m_h(f(\tilde{R}_h u),\phi_h^w) & \leq c h^2 ||(f(\tilde{R}_h u))^l||\ ||(\phi_h^w)^l||\\
 & \leq c h^2 (||f({R}_h u)-f(u)||+||f(u)||)\ ||(\phi_h^w)^l||.
\end{align*}
Using the Lipschitz continuity and the Ritz map error bound afterwards we obtain
\begin{align*}
m((f(\tilde{R}_h u))^l,(\phi_h^w)^l)-m_h(f(\tilde{R}_h u),\phi_h^w)
& \leq c h^2 (L||{R}_h u-u||+||f(u)||)\ ||(\phi_h^w)^l||\\
& \leq c h^2 (Lh^2||u||_2+||f(u)||)\ ||(\phi_h^w)^l||\\
& \leq c h^2 (||u||_2+||f(u)||)\ ||(\phi_h^w)^l||.
\end{align*}

\textit{(iii) Derivative of the defect in $w$.}\\
Because the order of using the lift operator and differentiation cannot be interchanged in general, we take equation (\ref{non_def}) without the first two terms on the right-hand side, differentiate it with respect to $t$ and add and subtract $
 m((\tilde{R}_h\pa_tw)^l,(\phi_h^w)^l)$ to obtain
\begin{align*}
m_h(\pa_t d_h^w,\phi_h^w)&= m_h(\tilde{R}_h\pa_tw,\phi_h^w)-m((\tilde{R}_h\pa_tw)^l,(\phi_h^w)^l)\nonumber\\ & \qquad+m(R_h\pa_t w-\pa_t w,(\phi_h^w)^l)\nonumber\\ & \qquad+m(f'(u)\pa_tu,(\phi_h^w)^l)-m_h(f'(\tilde{R}_h u)\tilde{R}_h \pa_tu,\phi_h^w).
\end{align*}
The first two terms on the right-hand side can be estimated as in the linear case. We consider the last term. We add and subtract $m(f'(R_h u)R_h \pa_tu,(\phi_h^w)^l)$, which then yields
\begin{align*}
m&(f'(u)\pa_tu,(\phi_h^w)^l)-m_h(f'(\tilde{R}_h u)\tilde{R}_h \pa_tu,\phi_h^w) \\ & = m(f'(u)\pa_tu,(\phi_h^w)^l)-m(f'(R_h u)R_h \pa_tu,(\phi_h^w)^l)\\ & \qquad+m(f'(R_h u)R_h \pa_tu,(\phi_h^w)^l)-m_h(f'(\tilde{R}_h u)\tilde{R}_h\pa_t u,\phi_h^w).
\end{align*}
Another term is inserted and we obtain
\begin{align}\label{part_a}
m&(f'(u)\pa_tu,(\phi_h^w)^l)-m_h(f'(\tilde{R}_h u)\tilde{R}_h \pa_tu,\phi_h^w) \nonumber\\ & = m(f'(u)\pa_tu,(\phi_h^w)^l)-m(f'(R_h u) \pa_tu,(\phi_h^w)^l)\nonumber\\ & \qquad+ m(f'(R_h u)\pa_tu,(\phi_h^w)^l)-m(f'(R_h u)R_h \pa_tu,(\phi_h^w)^l)\nonumber\\ & \qquad+m(f'(R_h u)R_h \pa_tu,(\phi_h^w)^l)-m_h(f'(\tilde{R}_h u)\tilde{R}_h\pa_t u,\phi_h^w)\nonumber\\
& = m((f'(u)-f'(R_h u) )\pa_tu,(\phi_h^w)^l)+ m(f'(R_h u)(\pa_tu-R_h \pa_tu),(\phi_h^w)^l)\nonumber\\ & \qquad+(m(f'(R_h u)R_h \pa_tu,(\phi_h^w)^l)-m_h(f'(\tilde{R}_h u)\tilde{R}_h\pa_t u,\phi_h^w)).
\end{align}
On the last term the geometric approximation error (\ref{geom_appr_error}) is applied 
\begin{align}\label{con_1}
m(f'(R_h u)R_h \pa_tu,(\phi_h^w)^l)-m_h(f'(\tilde{R}_h u)\tilde{R}_h\pa_t u,\phi_h^w)\leq c h^2 ||f'(R_h u)R_h \pa_tu||\  ||(\phi_h^w)^l||.
\end{align}
We  have to show that $||f'(R_h u)R_h \pa_tu||$ is smaller than some constant. Because we need a second-order bound, we obtain this term in the stronger norm. The next two pages we will deal with the lengthy estimate of this term.

 We decompose the norm in its bulk and surface part, which yields
\begin{align*}
||f'(R_h u)R_h \pa_tu|| & = ||f_\Omega '(R_h u)R_h \pa_tu||_{L^2(\Omega)}+||\nabla(f_\Omega'(R_h u)R_h \pa_tu)||_{L^2(\Omega)}\\
& \qquad + ||f_\Gamma'(\gamma (R_h u))\gamma( R_h \pa_tu)||_{L^2(\Gamma)}+||\nabla_\Gamma(f_\Gamma'(\gamma (R_h u))\gamma (R_h \pa_tu))||_{L^2(\Gamma)}.
\end{align*}
We consider the second term, the order terms can be bounded in a similar way. With the product rule and the triangle inequality it is obtained
\begin{align*}
||\nabla(f_\Omega'(R_h u)R_h \pa_tu)||_{L^2(\Omega)} & \leq ||\nabla(f_\Omega'(R_h u))\cdot R_h \pa_tu||_{L^2(\Omega)}+||f_\Omega'(R_h u)\cdot \nabla R_h \pa_tu||_{L^2(\Omega)}.
\end{align*}
By applying the generalized H\"older inequality it yields
\begin{align*}
||\nabla(f_\Omega'(R_h u)R_h \pa_tu)||_{L^2(\Omega)} & \leq ||\nabla(f_\Omega'(R_h u))||_{L^2(\Omega)} ||R_h \pa_tu||_{L^\infty(\Omega)}\\
& \qquad +||f_\Omega'(R_h u)||_{L^\infty(\Omega)} ||\nabla R_h \pa_tu||_{L^2(\Omega)}.
\end{align*}
We use that $||\nabla v||_{L^2(\Omega)}\leq || v||$. We add and subtract $f_\Omega'(u)$ and apply the Lipschitz continuity to obtain
\begin{align*}
||\nabla(f_\Omega'(R_h u)R_h \pa_tu)||_{L^2(\Omega)} & \leq (||f_\Omega'(R_h u)-f_\Omega'(u)||+||f_\Omega'(u)||)\ ||R_h \pa_tu||_{L^\infty(\Omega)}\\
& \qquad +(||f_\Omega'(R_h u)-f_\Omega'(u)||_{L^\infty(\Omega)} +||f_\Omega'(u)||_{L^\infty(\Omega)})|| R_h \pa_tu||\\
& \leq (L||R_h u-u||+||f_\Omega'(u)||)||R_h \pa_tu||_{L^\infty(\Omega)}\\
& \qquad +(L||R_h u-u||_{L^\infty(\Omega)} +||f_\Omega'(u)||_{L^\infty(\Omega)})|| R_h \pa_tu||.
\end{align*}

The interpolations ${I}_h \pa_t u$ and ${I}_hu$ are added and subtracted. The application of the triangle inequality then yields
\begin{align*}
||&\nabla(f_\Omega'(R_h u)R_h \pa_tu)||_{L^2(\Omega)}\\ & \leq (L||R_h u-u||+||f_\Omega'(u)||)( ||{R}_h \pa_tu-{I}_h \pa_t u||_{L^\infty(\Omega)} +||{I}_h \pa_t u||_{L^\infty(\Omega)})\\
& \qquad +(L(||R_h u-I_hu||_{L^\infty(\Omega)}+||u-I_hu||_{L^\infty(\Omega)})+||f_\Omega'(u)||_{L^\infty(\Omega)})||R_h \pa_tu||.
\end{align*}
To be able to apply the inverse estimate, we use the equivalence of the discrete and continuous norms and obtain
\begin{align*}
||&\nabla(f_\Omega'(R_h u)R_h \pa_tu)||_{L^2(\Omega)}\\ & \leq c(L||R_h u-u||+||f_\Omega'(u)||)( h^{-\frac{d}{2}} ||{R}_h \pa_tu-{I}_h \pa_t u||_{L^2(\Omega)} +||{I}_h \pa_t u||_{L^\infty(\Omega)})\\
& \qquad +c(L( h^{-\frac{d}{2}}||R_h u-I_hu||_{L^2(\Omega)}+||u-I_hu||_{L^\infty(\Omega)})+||f_\Omega'(u)||_{L^\infty(\Omega)})|| R_h \pa_tu||.
\end{align*}
 The norm equivalence is used again and we add and subtract $ u$ and $\pa_t u$. We apply the triangle inequality and obtain
\begin{align*}
||&\nabla(f_\Omega'(R_h u)R_h \pa_tu)||_{L^2(\Omega)}\\ & \leq c(L||R_h u-u||+||f_\Omega'(u)||)( h^{-\frac{d}{2}}( ||{R}_h \pa_tu- \pa_t u||_{L^2(\Omega)}+||{I}_h \pa_tu- \pa_t u||_{L^2(\Omega)})\\
& \qquad \qquad +||{I}_h \pa_t u-\pa_t u||_{L^\infty(\Omega)}+||\pa_t u||_{L^\infty(\Omega)})\\
& \qquad +c(L( h^{-\frac{d}{2}}(||R_h u-u||_{L^2(\Omega)}+||I_h u-u||_{L^2(\Omega)})+||u-I_hu||_{L^\infty(\Omega)})+||f_\Omega'(u)||_{L^\infty(\Omega)})\\
& \qquad\qquad(|| R_h \pa_tu-\pa_t u||+||\pa_t u||).
\end{align*}
The first- and second order error bounds for the Ritz error (\ref{ritz2}) and (\ref{ritz_v}), are applied and it yields
\begin{align*}
||&\nabla(f_\Omega'(R_h u)R_h \pa_tu)||_{L^2(\Omega)}\\ & \leq c(L||u||_2+||f_\Omega'(u)||)( h^{-\frac{d}{2}}( h^2||\pa_t u||_h+||{I}_h \pa_tu- \pa_t u||_{L^2(\Omega)})\\
& \qquad \qquad +||{I}_h \pa_t u-\pa_t u||_{L^\infty(\Omega)}+||\pa_t u||_{L^\infty(\Omega)})\\
& \qquad +c(L( h^{-\frac{d}{2}}(h^2||u||_2+||I_h u-u||_{L^2(\Omega)})+||u-I_hu||_{L^\infty(\Omega)})+||f_\Omega'(u)||_{L^\infty(\Omega)})\\
& \qquad\qquad(h||\pa_t u||_2+||\pa_t u||)\\
 & \leq c( h^{-\frac{d}{2}}( h^2||\pa_t u||_2+||{I}_h \pa_tu- \pa_t u||_{L^2(\Omega)}) +||{I}_h \pa_t u-\pa_t u||_{L^\infty(\Omega)}+||\pa_t u||_{L^\infty(\Omega)})\\
 & \qquad +c(( h^{-\frac{d}{2}}(h^2||u||_2+||I_h u-u||_{L^2(\Omega)})+||u-I_hu||_{L^\infty(\Omega)})+||f_\Omega'(u)||_{L^\infty(\Omega)}).
\end{align*}
We apply interpolation error estimates, Lemma \ref{inter}. For the bound in the $L^2$-norm we use \eqref{interpolation_bulk} and in the $L^\infty$-norm we use \eqref{inter_infty_bulk}. We obtain
\begin{align*}
||\nabla(f_\Omega'(R_h u)R_h \pa_tu)||_{L^2(\Omega)}
& \leq c( h^{-\frac{d}{2}}( h^2||\pa_t u||_2+h^2||\pa_t u||_2) +||\pa_t u||_{L^\infty(\Omega)})\\
& \qquad +c(( h^{-\frac{d}{2}}(h^2||u||_2+h^2||u||_2)+||u||_{L^\infty(\Omega)})+||f_\Omega'(u)||_{L^\infty(\Omega)}).
\end{align*}
For a dimension $d\leq 4$ the right-hand side above is smaller than a constant $c$ independent from $h$.\\

We consider again equation (\ref{part_a}) and will now bound $|(f'(u)-f'(R_h u) )\pa_tu|$ and $|f'(R_h u)(\pa_tu-R_h \pa_tu) |$. We use the generalized H\"older inequality for the first term and obtain
\begin{align*}
|(f'(u)-f'(R_h u) )\pa_tu| & = ||(f_\Omega '(u)-f_\Omega '(R_h u) )\pa_tu||_{L^2(\Omega)}\\
& \qquad + |(f_\Gamma '(\gamma u)-f_\Gamma '(\gamma (R_h u)) )\gamma (\pa_tu)||_{L^2(\Gamma)}\\
&  \leq  c(||f_\Omega '(u)-f_\Omega '(R_h u) ||_{L^2(\Omega)}||\pa_tu||_{L^\infty(\Omega)}\\
& + ||f_\Gamma '(\gamma u)-f_\Gamma '(\gamma (R_h u)) ||_{L^2(\Gamma)}||\gamma (\pa_tu)||_{L^\infty(\Gamma)}).
\end{align*}
The application of the Lipschitz continuity and the error bound for the Ritz map results in the desired second-order bound
\begin{align}\label{con_2}
|(f'(u)-f'(R_h u) )\pa_tu| \leq cL  h^2||u||_2(||\pa_tu||_{L^\infty(\Omega)}+||\gamma (\pa_tu)||_{L^\infty(\Gamma)})\leq ch^2.
\end{align}
For the second term we again start with using the generalized H\"older inequality, which then yields
\begin{align*}
|f'(R_h u)(\pa_tu-R_h \pa_tu)| & = ||f_\Omega'(R_h u)(\pa_tu-R_h \pa_tu)||_{L^2(\Omega)}\\ & \qquad+ |f_\Gamma'(\gamma(R_h u))\gamma(\pa_tu-R_h \pa_tu)||_{L^2(\Gamma)}\\
&  \leq  c(||\pa_tu-R_h \pa_tu ||_{L^2(\Omega)}||f_\Omega'(R_h u)||_{L^\infty(\Omega)}\\
& + ||\gamma(\pa_tu-R_h \pa_tu) ||_{L^2(\Gamma)}||f_\Gamma'(\gamma(R_h u))||_{L^\infty(\Gamma)}).
\end{align*}
On the first factor we apply the error estimate for  the Ritz map and obtain
\begin{align*}
|f'(R_h u)(\pa_tu-R_h \pa_tu)|
&  \leq  ch^2(||\pa_tu ||_2||f_\Omega'(R_h u)||_{L^\infty(\Omega)}+||\pa_tu ||_2||f_\Gamma'(\gamma(R_h u))||_{L^\infty(\Gamma)}).
\end{align*}
For bounding $||f_\Omega'(R_h u)||_{L^\infty(\Omega)} $ we add and subtract $f_\Omega'(u)$, use the Lipschitz continuity, which then yields
\begin{align*}
||f_\Omega'(R_h u)||_{L^\infty(\Omega)} &  \leq ||f_\Omega'(R_h u)-f_\Omega'(u)||_{L^\infty(\Omega)}+||f_\Omega'(u)||_{L^\infty(\Omega)}\\
&  \leq L||R_h u-u||_{L^\infty(\Omega)}+||f_\Omega'(u)||_{L^\infty(\Omega)}
\end{align*}
 We add and subtract $I_hu$, use the equivalence of norms and the inverse estimate, which yields
\begin{align*}
||f_\Omega'(R_h u)||_{L^\infty(\Omega)} &  \leq L||R_h u-u||_{L^\infty(\Omega)}+||f_\Omega'(u)||_{L^\infty(\Omega)}\\
& \leq L(||R_h u-I_hu||_{L^\infty(\Omega)}+||I_hu-u||_{L^\infty(\Omega)})+||f_\Omega'(u)||_{L^\infty(\Omega)}\\
& \leq cL(h^{-\frac{d}{2}}||R_h u-I_hu||_{L^2(\Omega)}+||I_hu-u||_{L^\infty(\Omega)})+||f_\Omega'(u)||_{L^\infty(\Omega)}.
\end{align*}
We add and subtract $u$. With Lipschitz continuity and the Ritz projection bound we obtain
\begin{align*}
||f_\Omega'(R_h u)||_{L^\infty(\Omega)}
& \leq cL(h^{-\frac{d}{2}}(||R_h u-u||_{L^2(\Omega)}+|| u-I_hu||_{L^2(\Omega)})\\
& \qquad+||I_hu-u||_{L^\infty(\Omega)})+||f_\Omega'(u)||_{L^\infty(\Omega)}\\
& \leq cL(h^{-\frac{d}{2}}(h^2||u||_2+|| u-I_hu||_{L^2(\Omega)})+||I_hu-u||_{L^\infty(\Omega)})+||f_\Omega'(u)||_{L^\infty(\Omega)}.
\end{align*} 
We apply the error bound for the interpolation error, both in the $L^2$-norm and the $L^\infty$-norm, which then yields
\begin{align*}
||f_\Omega'(R_h u)||_{L^\infty(\Omega)}
& \leq cL(h^{-\frac{d}{2}}(h^2||u||_2+h^2|| u||_2)+||u||_{L^\infty(\Omega)})+||f_\Omega'(u)||_{L^\infty(\Omega)}.
\end{align*} 
For a dimension $d\leq 4$ the term $||f_\Omega'(R_h u)||_{L^\infty(\Omega)}$ is bound independently from $h$.
The term $||f_\Gamma'(\gamma(R_h u))||_{L^\infty(\Gamma)} $ is bounded by a constant with the same proof.
This yields the second-order estimate
\begin{align}\label{con_3}
|f'(R_h u)(\pa_tu-R_h \pa_tu) | 
&\leq ch^2.
\end{align}
Altogether we obtain for the derivative of the defect in $w$
\begin{align*}
m_h(\pa_t d_h^w,\phi_h^w)&= (m_h(\tilde{R}_h\pa_tw,\phi_h^w)-m((\tilde{R}_h\pa_tw)^l,(\phi_h^w)^l))\nonumber\\ & \qquad+m(R_h\pa_t w-\pa_t w,(\phi_h^w)^l)\nonumber\\ & \qquad +m((f'(u)-f'(R_h u) )\pa_tu,(\phi_h^w)^l)+ m(f'(R_h u)(\pa_tu-R_h \pa_tu),(\phi_h^w)^l)\nonumber\\ & \qquad+(m(f'(R_h u)R_h \pa_tu,(\phi_h^w)^l)-m_h(f'(\tilde{R}_h u)\tilde{R}_h\pa_t u,\phi_h^w)).
\end{align*}
The first and the last term are both estimated with the geometric approximation error, which yields 
\begin{align*}
m_h(\pa_t d_h^w,\phi_h^w)&\leq ch^2 ||{R}_h\pa_tw||\ ||(\phi_h^w)^l||+m(R_h\pa_t w-\pa_t w,(\phi_h^w)^l)\nonumber\\ & \qquad +m((f'(u)-f'(R_h u) )\pa_tu,(\phi_h^w)^l)+ m(f'(R_h u)(\pa_tu-R_h \pa_tu),(\phi_h^w)^l)\nonumber\\ & \qquad+ch^2||f'(R_h u)R_h \pa_tu||\ ||(\phi_h^w)^l||.
\end{align*}
We add and subtract $\pa_t w$ and apply the Cauchy-Schwarz inequality to obtain
\begin{align*}
m_h(\pa_t d_h^w,\phi_h^w)&\leq ch^2 (||{R}_h\pa_tw- \pa_t w||+||\pa_t w||)\ ||(\phi_h^w)^l||+|R_h\pa_t w-\pa_t w|\ |(\phi_h^w)^l|\nonumber\\ & \qquad +|(f'(u)-f'(R_h u) )\pa_tu|\ |(\phi_h^w)^l|+ |f'(R_h u)(\pa_tu-R_h \pa_tu)|\ |(\phi_h^w)^l|\nonumber\\ & \qquad+ch^2||f'(R_h u)R_h \pa_tu||\ ||(\phi_h^w)^l||.
\end{align*}

The Ritz error bounds \eqref{ritz2} and \eqref{ritz_v} yields
\begin{align*}
m_h(\pa_t d_h^w,\phi_h^w)&\leq ch^2 (h||\pa_t w||_2+||\pa_t w||)\ ||(\phi_h^w)^l||+ch^2||\pa_t w||\ |(\phi_h^w)^l|\nonumber\\ & \qquad +|(f'(u)-f'(R_h u) )\pa_tu|\ |(\phi_h^w)^l|+ |f'(R_h u)(\pa_tu-R_h \pa_tu)|\ |(\phi_h^w)^l|\nonumber\\ & \qquad+ch^2||f'(R_h u)R_h \pa_tu||\ ||(\phi_h^w)^l||.
\end{align*}

With the bounds shown above \eqref{con_1}, \eqref{con_2} and \eqref{con_3} we obtain
\begin{align*}
m_h(\pa_t d_h^w,\phi_h^w)&\leq ch^2 (h||\pa_t w||_2+||\pa_t w||)\ ||(\phi_h^w)^l||+ch^2||\pa_t w||\ |(\phi_h^w)^l|\nonumber\\ & \qquad +ch^2|(\phi_h^w)^l|+ ch^2\ |(\phi_h^w)^l|\nonumber+ch^2||(\phi_h^w)^l||\\
& \leq c h^2 |\phi_h^w||.
\end{align*}

For the defect in the dual norm it yields 
\begin{align*}
||\pa_td_h^w||_{\ast,h} = \sup_{\phi_h^w \in V_h\setminus \lbrace 0 \rbrace}\frac{m_h(\pa_td_h^w,\phi_h^w)}{||\phi_h^w||_h}
\leq \sup_{\phi_h^w \in V_h\setminus \lbrace 0 \rbrace}\frac{c h^2 ||\phi_h^w||}{||\phi_h^w||_h} = ch^2\
\end{align*}

\end{proof}
\section{Convergence}

	\begin{proof}[Proof of Theorem \ref{non_lin_convergence}]
	
		The proof mainly combines the previous two sections about stability and consistency and is similar to the linear variant.

We use the same decomposition as in the linear case:
$$u_h^l-u = (u_h-\tilde{R}_h u)^l+(R_hu-u) ,$$
$$w_h^l-w = (w_h-\tilde{R}_h w)^l+(R_hw-w),$$
$$\partial_t u_h^l-\partial_t u = (\partial_t u_h-\tilde{R}_h \partial_t u)^l+(R_h \partial_t u-\partial_t u).$$

The second terms are estimated with (\ref{ritz2}) for  second-order estimate in the $|\cdot|$-norm and with (\ref{ritz_v}) for a first-order estimate in the $||\cdot||$-norm.

 To estimate the first terms we define the errors $e_h^u:=\tilde{R}_hu-u_h$ and $ e_h^w:=\tilde{R}_hw-w_h$. We will apply the stability result, Proposition \ref{non_energy_est_semidiscr}, therefore we review the prerequisites  of this proposition:
  \eqref{cond_ritz} holds true with the bounds for the Ritz map, 
  \eqref{cond_def} is satisfied due to the Proposition \ref{nl_consis},
  \eqref{cond_init} is fulfilled because of assumption (\ref{initial}).

The assumption of $e_h^u(0)=0$ and $e_h^w(0)=0$ implies $d_h^u(0)=0$ and $d_h^w(0)=0$. Applying Proposition \ref{non_energy_est_semidiscr} yields the following stability estimate for the errors
\begin{align*}
||&e_h^u(t)||_h^2+||e_h^w(t)||_h^2+\int_0^t||\partial_te_h^u(s)||_h^2\mathrm{d}s+\int_0^t||e_h^w(s)||_h^2\mathrm{d}s  \\  & 
\leq c \Big( ||d_h^u(t)||_{\ast,h}^2+\int_0^t||d_h^u(s)||_{\ast,h}^2 \mathrm{d}s+\int_0^t||\partial_td_h^u(s)||_{\ast,h}^2\mathrm{d}s\\
& \qquad+\int_0^t||d_h^w(s)||_{\ast,h}^2\mathrm{d}s+\int_0^t||\partial_td_h^w(s)||_{\ast,h}^2\mathrm{d}s\Big),
\end{align*}
with $c$  independent from $h$.\\
Since we chose $u_h^\ast=\tilde{R}_hu$ and $w_h^\ast=\tilde{R}_hw$, the consistency result, Proposition \ref{pr_consist}, applies and we obtain
\begin{align*}
||&e^u_h(t)||_h^2+||e^w_h(t)||_h^2+\int_0^t||\partial_te^u_h(s)||_h^2\mathrm{d}s+\int_0^t||e^w_h(s)||_h^2\mathrm{d}s
\leq ch^4 .
\end{align*}
Because it is $|\cdot |_h\leq||\cdot||_h$, we obtain as well
\begin{align*}
|&e^u_h(t)|_h^2+|e^w_h(t)|_h^2+\int_0^t|\partial_te^u_h(s)|_h^2\mathrm{d}s+\int_0^t|e^w_h(s)|_h^2\mathrm{d}s
\leq ch^4 .
\end{align*}
This inequalities, the equivalence of discrete and continuous norms and the bounds for the Ritz error yield	the statements of Theorem \ref{non_lin_convergence}. 
\end{proof}

\chapter{Time discretization}

To finally solve the Cahn--Hilliard equation with dynamic boundary conditions numerically we will discretize in time using the backward difference formulae (BDF). For the linear case we use the usual BDF method, for the nonlinear case the linearly implicit method.
The $k$-step BDF method is determined by the coefficients of the generating polynomial $\delta(\xi)= \sum_{l=0}^k \delta_l i^l= \sum_{l=1}^k \frac{1}{l}(1-\xi)^l$. The classical BFD method is zero-stable for $k\leq 6$ and has order $k$, see \textit{Hairer\&Wanner (1996)}, \cite[Chapter ~V]{HW96}. The method is $A$-stable for $k\leq 2$ and $A(\alpha)$-stable for $3\leq k\leq 6$.
\section{Linear case}\label{time_lin}

We recall the linear Cahn--Hilliard/Cahn--Hilliard coupling \eqref{Cahn_H_bulk}--\eqref{Cahn_H_suface} and add another inhomogeneity to obtain
\begin{subequations}\label{ch_extra_f}
	\begin{align}
	\partial_t u(x,t) &= \Delta w(x,t)+ f_1^\Omega(x,t)\qquad&\text{in}\  \Omega\times[0,T]\\
	w(x,t) &=  -\Delta u(x,t) + f_2^\Omega(x,t) \qquad&\text{in}\  \Omega\times[0,T],
	\end{align}
\end{subequations}
with dynamic boundary conditions
\begin{subequations}\label{ch_extra_f_surface}
	\begin{align}
	\partial_t u(x,t) &= \Delta_\Gamma w(x,t) - \partial_\nu w(x,t)+f_1^\Gamma (x,t)\qquad &\text{on}\  \Gamma\times [0,T] \\
	w(x,t) &=  -\Delta_\Gamma u(x,t)  + \partial_\nu u(x,t)+ f_2^\Gamma (x,t) \qquad &\text{on}\  \Gamma\times [0,T].
	\end{align}
\end{subequations}
We recall the matrix-vector-formulation \eqref{m_v_formulation} 
\begin{subequations}\label{n_m_y}
	\begin{eqnarray}
	\bfM\dot{\text{\textbf{u}}}(t) + \bfA\bfw(t) &=& \bfb_1(t)\\
	\bfM \bfw(t) - \bfA\bfu(t) &=& \bfb_2(t),
	\end{eqnarray}
\end{subequations}
with $\bfb_1^{(i)}(t):= m_h(\tilde{I}_h f_1(t),\phi_i)$ and $\bfb_2^{(i)}(t):= m_h(\tilde{I}_h f_2(t),\phi_i)$.

For a time step size $\tau > 0$ we set $t^n = n\tau \leq T$.
Applying the k-step BDF method on \eqref{n_m_y} yields the following linear equation system
\begin{eqnarray*}
	\frac{1}{\tau}\bfM\sum_{j=0}^k\delta_j \bfu(t^{n-j}) + \bfA\bfw(t^n) &=& \bfb_1(t^n)\\
	\bfM \bfw(t^n) - \bfA\bfu(t^n) &=& \bfb_2(t^n)
\end{eqnarray*}
which, by setting $\bfu^n:=\bfu(t^n)$ and $\bfw^n:=\bfw(t^n)$, can be rewritten as follows
\begin{equation*}
\begin{pmatrix}
\frac{\delta_0}{\tau}\bfM & \bfA\\
-\bfA & \bfM
\end{pmatrix}
\begin{pmatrix}
\bfu^n \\
\bfw^n
\end{pmatrix}
= 
\begin{pmatrix}
\bfb_1(t^n)-\frac{1}{\tau}\bfM\sum_{j=1}^k \delta_j \bfu^{n-j}\\
\bfb_2(t^n)
\end{pmatrix}.
\end{equation*}
This linear system has to be solved for every time step. The starting values $\bfu^0,\ldots,\bfu^{k-1}$ and $\bfw^0,\ldots,\bfw^{k-1}$ are assumed to be given, e.g. obtained by a Runge--Kutta method. 

\section{Nonlinear case}\label{time_nl}

We recall the nonlinear Cahn--Hilliard/Cahn--Hilliard coupling \eqref{non_linear_cahn_bulk}--\eqref{non_linear_cahn_surface}. We add inhomogeneities $f_1^\Omega, f_2^\Omega, f_1^\Gamma$ and $f_2^\Gamma$, rename the nonlinearities as $F_\Omega:\R\to\R$ and $F_\Gamma:\R\to\R$ and  to obtain
\begin{subequations}\label{nch_extra_f}
	\begin{align}
	\partial_t u(x,t) &= \Delta w(x,t)+ f_1^\Omega(x,t)\qquad&\text{in}\  \Omega\times[0,T]\\
	w(x,t) &=  -\Delta u(x,t) + f_2^\Omega(x,t) +F_\Omega(u(x,t))\qquad&\text{in}\  \Omega\times[0,T],
	\end{align}
\end{subequations}
with dynamic boundary conditions
\begin{subequations}\label{nch_extra_f_s}
	\begin{align}
	\partial_t u(x,t) &= \Delta_\Gamma w(x,t) - \partial_\nu w(x,t)+f_1^\Gamma (x,t)\qquad &\text{on}\  \Gamma\times [0,T] \\
	w(x,t) &=  -\Delta_\Gamma u(x,t)  + \partial_\nu u(x,t)+ f_2^\Gamma (x,t)+F_\Gamma(u(x,t)) \qquad &\text{on}\  \Gamma\times [0,T].
	\end{align}
\end{subequations}
The matrix-vector formulation then appears as
\begin{subequations}\label{nn_m_y}
	\begin{eqnarray}
	\bfM\dot{\text{\textbf{u}}}(t) + \bfA\bfw(t) &=& \bfb_1(t)\\
	\bfM \bfw(t) - \bfA\bfu(t) &=& \bfb_2(t)+ \bfF(\bfu(t)),
	\end{eqnarray}
\end{subequations}
with $\bfb_1^{(i)}(t):= m_h(\tilde{I}_h f_1(t),\phi_i)$, $\bfb_2^{(i)}(t):= m_h(\tilde{I}_h f_2(t),\phi_i)$ and $\bfF(\bfu(t))_i= m_h(F(u),\phi_i) $, where $F=(F_\Omega,F_\Gamma)$.

To have the nonlinearity  given explicitly in each time step we approximate $\bfu(t^n)$, where it appears as an argument of the nonlinearity, by the extrapolated value $ \sum_{j=0}^{k-1} \gamma_j \bfu ^{n-j-1}\approx \bfu(t^n)$. The coefficients $\gamma_j$ are defined as the coefficients of the polynomial $\gamma(\xi)=\sum_{l=0}^{k-1} \gamma_l \xi^l = (1-(1-\xi)^k)/\xi$.
The linearly implicit BDF method applied on \eqref{nn_m_y} is described by following system
\begin{eqnarray*}
	\frac{1}{\tau}\bfM\sum_{j=0}^k\delta_j \bfu(t^{n-j}) + \bfA\bfw(t^n) &=& \bfb_1(t^n)\\
	\bfM \bfw(t^n) - \bfA\bfu(t^n) &=& \bfb_2(t^n) + \bfF\Big(\sum_{j=0}^{k-1} \gamma_j \bfu ^{n-j-1}\Big)
\end{eqnarray*}
which can be rewritten as follows
\begin{equation*}
\begin{pmatrix}
\frac{\delta_0}{\tau}\bfM & \bfA\\
-\bfA & \bfM
\end{pmatrix}
\begin{pmatrix}
\bfu^n \\
\bfw^n
\end{pmatrix}
= 
\begin{pmatrix}
\bfb_1(t^n)-\frac{1}{\tau}\bfM\sum_{j=1}^k \delta_j \bfu^{n-j}\\
\bfb_2(t^n)+ \bfF(\sum_{j=0}^{k-1} \gamma_j \bfu ^{n-j-1})
\end{pmatrix}.
\end{equation*}
This linear system has to be solved for every time step. The starting values $\bfu^0,\ldots,\bfu^{k-1}$ and $\bfw^0,\ldots,\bfw^{k-1}$ are assumed to be given, e.g. obtained by a Runge--Kutta method.
The linearly implicit BDF method remains to be zero-stable and has still order $k$, cf. \textit{Akrivis \&Lubich (2015)}, \cite{AL}.

\chapter{Numerical experiments}

In this chapter we present our numerical results for linear and nonlinear variants of the Cahn--Hilliard equation with dynamic Cahn--Hilliard boundary conditions, which illustrate the proven convergence bounds, Theorem \ref{convergence} and \ref{non_lin_convergence}. 

For spatial discretization the implementation uses linear finite elements and for time discretization we use the BDF3 method as presented in the former chapter. As the domain we use the two-dimensional unit disk, and meshes that are generated via DistMesh \cite{distmesh}. We will show results for different mesh sizes and different time step sizes. We choose the inhomogeneities, $f_1^\Omega, f_2^\Omega, f_1^\Gamma$ and $f_2^\Gamma$, such that the solution of \eqref{ch_extra_f}--\eqref{ch_extra_f_surface} respectively \eqref{nch_extra_f}--\eqref{nch_extra_f_s} is $(u, w)= (e^{-t}xy,e^{-t}xy)$.
For the nonlinear case we consider the Cahn--Hilliard equation with double-well potential, both in the bulk and on the surface, therefore we choose the nonlinearity $F(u)=u^3-u$.

We consider the error between the exact solution at time $t^N=1$ and the approximated solution we computed, via the methods described in Sections \ref{time_lin} and \ref{time_nl}, respectively. In Figures \ref{1test} and \ref{2test} we present the errors in the $L^2$-norm on $L^2(\Omega)\times L^2(\Gamma)$ on the left-hand sides and the errors in the $H^1$-norm on $\lbrace v\in H^1(\Omega) : \gamma v\in H^1(\Gamma)\rbrace $ on the right-hand sides. The errors are both shown a double logarithmic plot and are plotted against the different mesh sizes. We used meshes in eight different refinements, the number of nodes we had on the unit disk where $2^i\cdot10 $ for $i=1,\ldots, 8$.  We use time step sizes $\tau = (0.025, 0.0125, 0.005, 0.0025)$, the lines for the different step sizes are marked by different symbols. In all plots are dashed, quadratic reference lines, in the right subplots there is a linear reference line as well. For all error lines in Figure \ref{1test} and \ref{2test} we observe two regions: One region where the error lines match the quadratic convergence rate, since the spatial discretization error dominates, and one region where the lines flatten out, since the time discretization error starts to dominate.  For both left-hand side subplots, i.e. the $L^2$-errors, the convergence rate matches with the convergence rates we proved in Theorem \ref{convergence} and \ref{non_lin_convergence}. For the errors in the $H^1$-norm we observe a better convergence than we expected, which was $\mathcal{O}(h^2)$.

\section*{Evolution plots of the Cahn--Hilliard equation with dynamic boundary conditions}

An illustrate of the phenomena of phase separation described  by the Cahn--Hilliard equation with dynamic Cahn--Hilliard boundary conditions is shown in Figure \ref{sol}. We considered a nonlinear variant of the Cahn--Hilliard equation with potentials $W(u)=10(u^2-1)^2$. We again used the linear finite element method and the linearly implicit version of the BDF3 method. For our plots shown in the figure we generated random initial data $u_0\in \lbrace -1,1\rbrace$, choose a grid with 640 nodes, a time step size $\tau = 0.00125$ and choose the domain as the disk with radius 10. The colorbar in Figure \ref{sol} is valid for every subplot. Each subplot shows the solution $u$ of our problem at a certain time $t\in [0,3]$, we can clearly observe the phase separation.

\vspace{3cm}

The Matlab files containing our implementation for solving the Cahn--Hilliard equation with dynamic Cahn--Hilliard boundary conditions and generating the presented plots (Figures \ref{1test}, \ref{2test} and \ref{sol}) can be found at: https://gitlab.com/paulaharder/masterthesis.

\clearpage
\begin{figure}[htbp]
	\includegraphics[width=1\textwidth,trim={115 10 150 0},clip]{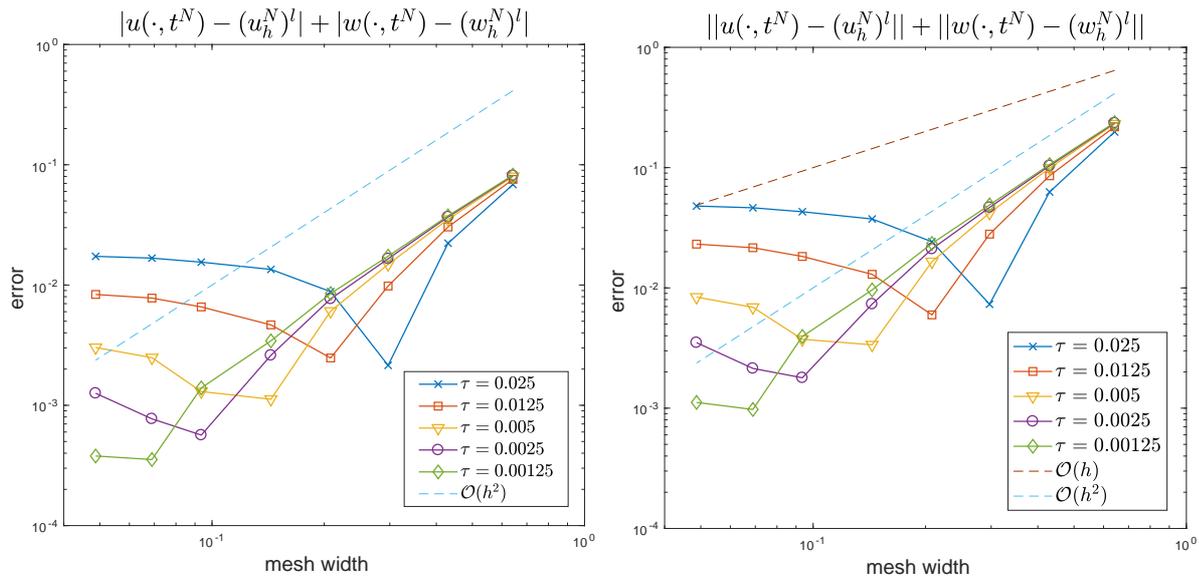}
	\caption{Spatial convergence for the linear Cahn--Hilliard equation with dynamic boundary conditions \eqref{ch_extra_f}--\eqref{ch_extra_f_surface}}
	\label{1test}
\end{figure}
\begin{figure}[htbp]
	\includegraphics[width=1\textwidth,trim={110 0 150 0},clip]{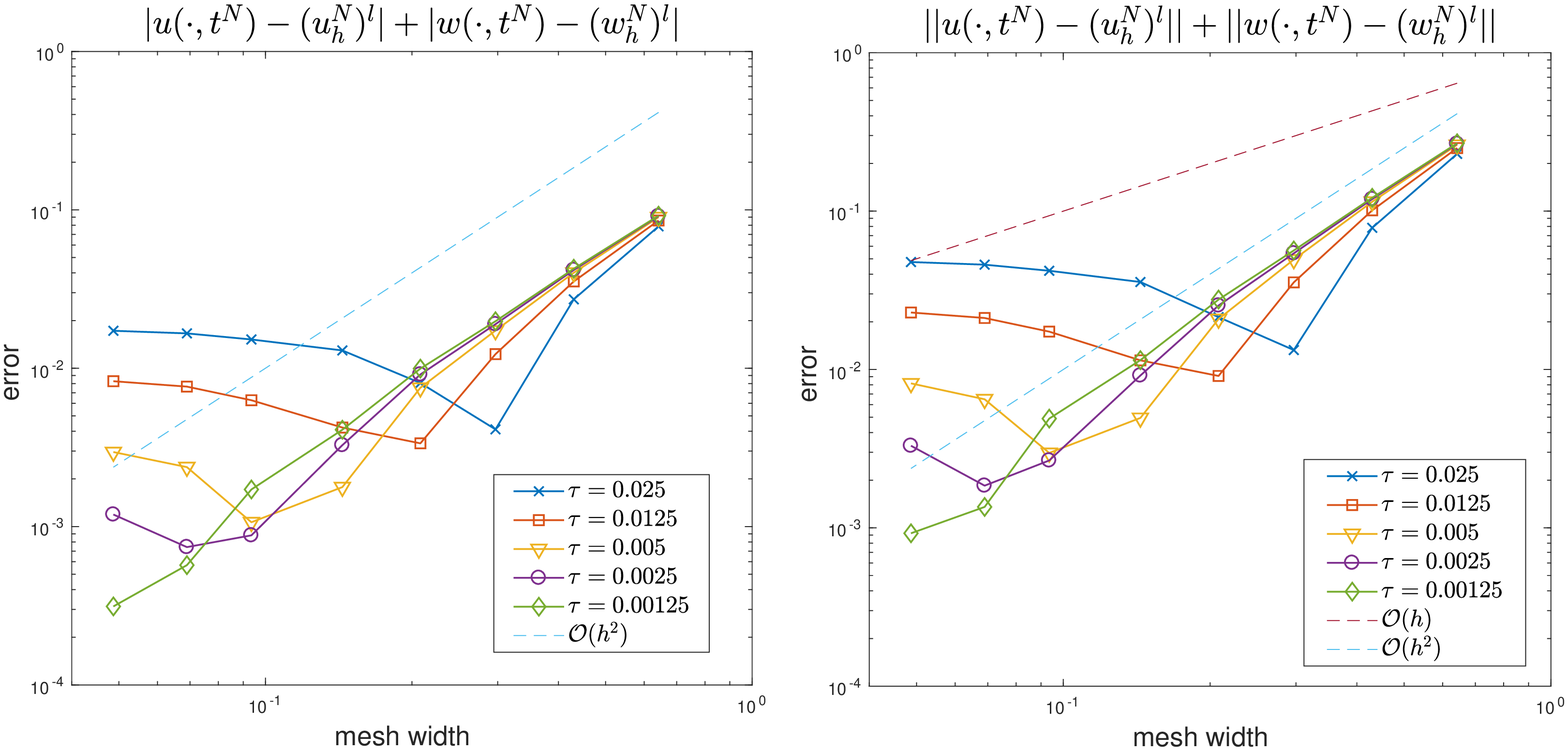}
	\caption{Spatial convergence for the nonlinear Cahn--Hilliard equation with dynamic boundary conditions \eqref{nch_extra_f}--\eqref{nch_extra_f_s} with potentials $\frac{1}{4}(u^2-1)^2$}
	\label{2test}
\end{figure}

\clearpage
\begin{figure}[htbp]
	\includegraphics[width=1\textwidth,trim={200 0 500 0},clip]{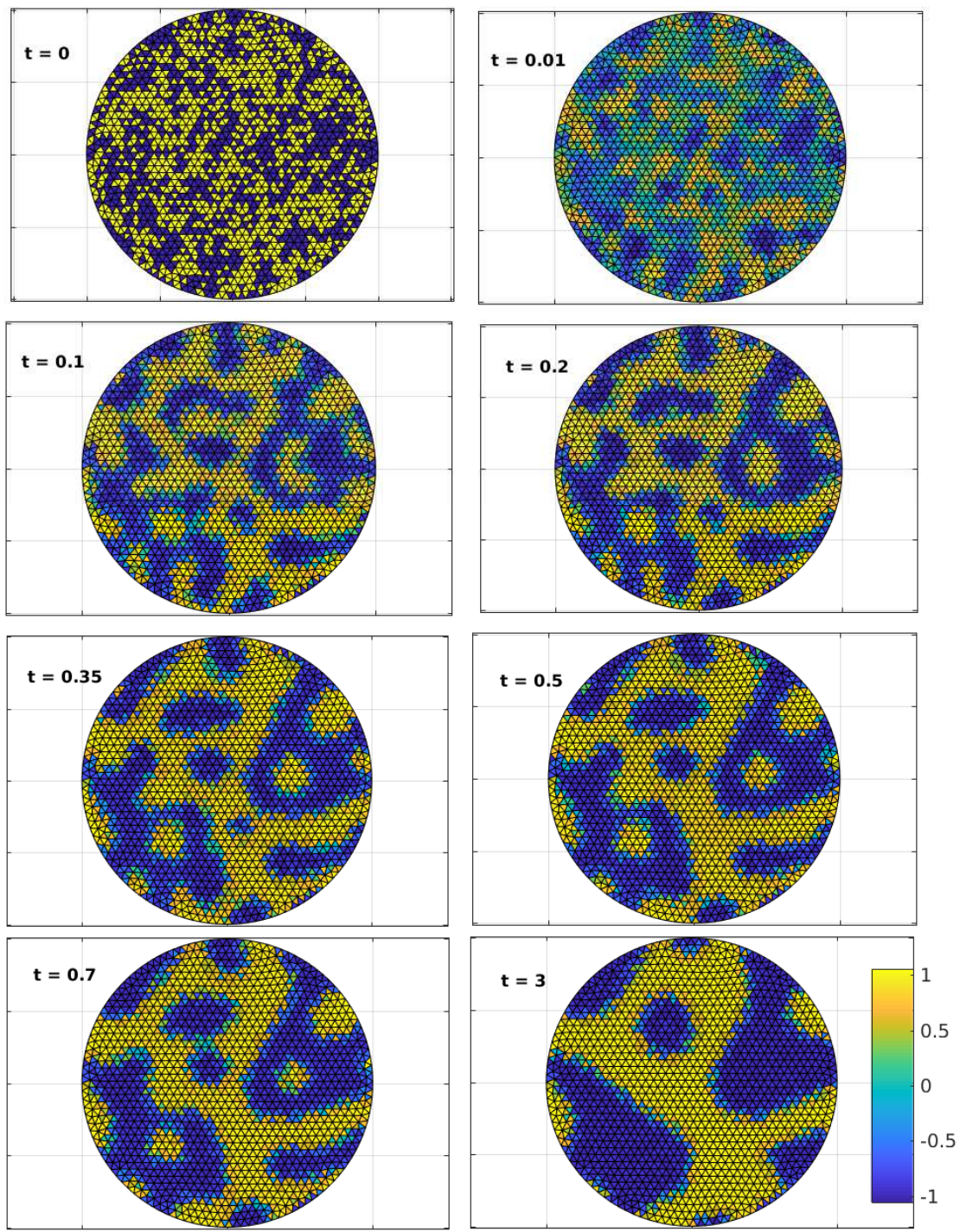}	\caption{Evolution the nonlinear Cahn--Hilliard equation with dynamic boundary conditions \eqref{nch_extra_f}--\eqref{nch_extra_f_s} with potentials $10(u^2-1)^2$}
	\label{sol}
\end{figure}

\clearpage

\newpage
\thispagestyle{empty}

\cleardoublepage
\bibliographystyle{acm}
\bibliography{mt_ref}
\addcontentsline{toc}{chapter}{Bibliography}
\end{document}